\documentclass[12pt, a4paper]{article}
\usepackage{latexsym}
\usepackage{graphicx}
\usepackage{amsmath} 
\usepackage{amsfonts} 
\usepackage{amsthm} 
\usepackage{amssymb}
\usepackage{enumerate}
\usepackage{tikz}

\input amssym.def
\input amssym.tex

\usepackage{authblk}

\usepackage{hyperref}              
\usepackage[capitalise]{cleveref}  
\hypersetup{colorlinks=true,
   citecolor=blue,
   filecolor=blue,
   linkcolor=blue,
   urlcolor=blue
}
\usepackage{url}

\theoremstyle{definition}

\newtheorem{theorem}{Theorem}[section]
\newtheorem{lemma}[theorem]{Lemma}

\newtheorem{construction}[theorem]{Construction}

\newtheorem{defn}[theorem]{Definition}
\newtheorem{example}[theorem]{Example}
\newtheorem{remark}[theorem]{Remark}

\def\qed{\hfill$\Box$\vspace{12pt}}

\long\def\delete#1{}

\newcommand{\cay}{\mathrm{Cay}}

\def\b{\mathbf b}

\def\ZZZ{\mathbb{Z}}

\def\b0{\mathbf{0}}

\def\De{\Delta}

\def\Ga{\Gamma}

\def\Si{\Sigma}

\def\a{\alpha}
\def\b{\beta}

\def\d{\delta}

\def\z{\zeta}

\def\s{\sigma}

\def\t{\tau}

\def\Cay{{\rm Cay}}

\def\gcd{{\rm gcd}}

\newcommand{\ph}{\pi_H}

\newcommand{\baf}{\bar{f}}
\newcommand{\bag}{\bar{g}}

\newcommand{\GG}{\mathcal{G}}

\usepackage{xcolor}
\usepackage[normalem]{ulem}

\begin{document}

\title{Perfect codes in circulant graphs of degree $p^l-1$}

\author[1]{Xiaomeng Wang\thanks{E-mail: \texttt{wangxm2015@lzu.edu.cn}}}
\author[2]{Oriol Serra\thanks{E-mail: \texttt{oriol.serra@upc.edu}}}
\author[1]{Shou-Jun Xu\thanks{E-mail: \texttt{shjxu@lzu.edu.cn}}}
\author[3]{Sanming Zhou\thanks{E-mail: \texttt{sanming@unimelb.edu.au}}}

\affil[1]{\small School of Mathematics and Statistics, Gansu Center for Applied Mathematics, Lanzhou University, Lanzhou, Gansu 730000, China}
\affil[2]{\small Department of Mathematics, Universitat Polit\`ecnica de Catalunya, Barcelona, Spain}
\affil[3]{\small School of Mathematics and Statistics, The University of Melbourne, Parkville, VIC 3010, Australia}

\date{}
\maketitle
\openup 0.5\jot

\begin{abstract}
A perfect code in a graph is an independent set of the graph such that every vertex outside the set is adjacent to exactly one vertex in the set. A circulant graph is a Cayley graph of a cyclic group. In this paper we study perfect codes in circulant graphs of degree $p^l - 1$, where $p$ is a prime and $l \ge 1$. We obtain a necessary and sufficient condition for such a circulant graph to admit perfect codes, give a construction of all such circulant graphs which admit perfect codes, and prove a lower bound on the number of distinct perfect codes in such a circulant graph. This extends known results for the  case $l=1$ and provides insight on the general problem on the existence and structure of perfect codes in circulant graphs.


\end{abstract}


\section{Introduction}

All graphs considered in this paper are finite, undirected and simple, and all groups considered are finite. Let $\Gamma = (V(\Ga), E(\Ga))$ be a graph and $e \ge 1$ an integer. A subset $D$ of $V(\Ga)$ is called \cite{Biggs73, Z15} a \emph{perfect $e$-code} in $\Gamma$ if every vertex of $\Ga$ is at distance no more than $e$ to exactly one vertex in $D$. In what follows a perfect $1$-code is simply called a \emph{perfect code}. So a subset $D$ of $V(\Ga)$ is a perfect code in $\Gamma$ if and only if $D$ is an independent set of $\Ga$ and every vertex in $V(\Ga) \setminus D$ has exactly one neighbour in $D$, which holds if and only if $\{\Gamma[x]: x \in D\}$ is a partition of $V(\Ga)$, where $\Ga[x]$ is the closed neighbourhood of $x$ in $\Ga$. A perfect code in a graph is also called an \emph{efficient dominating set} \cite{DS03} or \emph{independent perfect dominating set} \cite{L01}. A graph is distance-transitive \cite{Biggs73} if any pair of vertices can be permuted to any other pair of vertices at the same distance by an automorphism of the graph. Given a group $G$ and an inverse-closed subset $S$ of $G$ excluding the identity element, the \emph{Cayley graph} $\cay(G, S)$ of $G$ with \emph{connection set} $S$ is the group with vertex set $G$ such that $x, y$ are adjacent if and only if $xy^{-1} \in S$. Cayley graphs of cyclic groups are widely known as \emph{circulant graphs}.

The study of perfect $e$-codes in graphs began with N. Biggs \cite{Biggs73} who among other things generalized the celebrated Lloyd's theorem in coding theory to all distance-transitive graphs. Since the Hamming graph $H(d, q)$ is distance-transitive, perfect $e$-codes in distance-transitive graphs can be considered as a generalization of perfect $e$-codes in classical coding theory. On the other hand, since $H(d, q)$ is the Cayley graph of $\ZZZ_q^d$ with respect to the set of elements of $\ZZZ_q^d$ with exactly one nonzero coordinate, perfect $e$-codes in Cayley graphs can be viewed as another generalization of perfect $e$-codes in coding theory. Also, perfect codes in Cayley graphs are essentially factorizations and tilings of the underlying group \cite{hfz18, Sands03, SS09}. Because of these connections, perfect codes in Cayley graphs have attracted considerable attention in recent years. See \cite[Section 1]{fhz17} and \cite[Section 1]{hfz18} for short surveys of related results and background information. See also \cite{cwx20, kak23, wwz20, zz21, zz21a} for recent results on subgroup perfect codes in Cayley graphs, where a perfect code in a Cayley graph is a \emph{subgroup perfect code} if it is a subgroup of the underlying group.

In this paper we study perfect codes in circulant graphs of degree $p^l - 1$, where $p$ is a prime and $l \ge 1$. We obtain a necessary and sufficient condition for such a circulant graph to admit perfect codes, give a construction of all such circulant graphs admitting perfect codes, and obtain a lower bound on the number of distinct perfect codes in such a circulant graph. The characterization extends to the case $l\ge 1$ the known one for the case $l=1$ and provides a relatively rich set of perfect codes in circulant graphs. A particular motivation of this paper is also to provide a deeper insight on the existence and structure of perfect codes. We note that a subset $D$ of an abelian group $G$ containing $0$ is a perfect code in $\cay(G, S)$ if and only if $(S_0, D)$ is a tiling of $G$, that is, every element of $G$ can be expressed in a unique way as $a + b$ with $a \in S_0$ and $b \in D$. So our results in this paper can be stated in terms of tilings of $\ZZZ_n$. 

\subsection{Some history}

Before presenting our results, let us give a brief account of known results on perfect codes in circulant graphs. It was observed in \cite[Remark 1]{ob07} (see also \cite[Lemma 2.3]{fhz17}) that a connected circulant graph $\Cay(\ZZZ_n, S)$ of order $n \ge 4$ admits a perfect code if $|S_0|$ divides $n$ and $s \not \equiv s'$ (mod $|S_0|$) for distinct $s, s' \in S_0$, and in this case the unique subgroup of $\ZZZ_n$ with order $n/|S_0| $ is a perfect code in $\Cay(\ZZZ_n, S)$. This sufficient condition has been proved to be necessary in the following cases:
\begin{enumerate}
\item[(i)] $|S_0| = p$ for an odd prime $p$ (\cite[Theorem 3.1]{deng17} and \cite[Theorem 1.1]{fhz17});
\item[(ii)] $|S_0| = p^l$ with $p$ a prime and $l$ the exponent of $p$ in $n$ (\cite[Theorem 3.4]{deng17} and \cite[Theorem 1.2]{fhz17});
\item[(iii)] $|S_0| = pq$ for some (not necessarily distinct) prime divisors $p, q$ of $n$ with $pq < n$ such that $n/pq$ and $pq$ are coprime and $S_0$ is periodic (\cite[Theorem 3.2]{deng17});
\item[(iv)] $n \in \{p^{k}, p^{k}q, p^{2}q^{2}, pqr, p^{2}qr, pqrs: p, q, r, s \text{ are primes and } k \ge 1\}$, $|S_0| < n$, and $n/|S_0|$ and $|S_0|$ are coprime (\cite[Theorem 3.5]{deng17});
\item[(v)] $n/|S_0|$ is a prime and $S_0$ is aperiodic (\cite[Theorem 3.2]{deng14}).
\end{enumerate}

In \cite{ob07}, Obradovi\'{c}, Peters and Ru\v{z}i\'{c} gave a necessary and sufficient condition for the existence of perfect codes in a circulant graph of degree $3$ or $4$, and showed that under this condition all perfect codes are equally spaced (that is, they are cosets of some subgroup perfect codes). In a recent paper, Kwon, Lee and Sohn \cite{kls22} gave a classification of all perfect codes in any circulant graph of degree $5$. In \cite[Theorem 3.6]{cm13}, Tamizh Chelvam and Mutharasu proved that a proper subgroup $H$ of $\ZZZ_n$ is a perfect code in some connected Cayley graph of $\ZZZ_n$ if and only if either $|H|$ or $n/|H|$ is odd. A few years later Huang, Xia and Zhou proved that this result is a special case (\cite[Corollary 2.8]{hfz18}) of a general result \cite[Theorem 2.7]{hfz18} in the framework of studying when a given subgroup is a perfect code in some Cayley graph of the group.

In \cite{kum13}, Kumar and  MacGillivray  constructed infinitely many examples of circulant graphs which admit non-equally spaced perfect codes, and they gave a characterization of circulant graphs which admit perfect codes of size $2$ or $3$ using the wreath product of a circulant graph and a complete graph. Their research prompted the following question \cite{kum13}: Is it true that in any circulant graph admitting perfect codes either all perfect codes are equally spaced or the circulant graph is the wreath product of a smaller circulant graph admitting perfect codes and a complete graph on at least two vertices? Our work in this paper shows that the answer to this question is negative; a counterexample will be given in Example~\ref{ex:ex3.4}.

In \cite{Z15}, the fourth author studied two families of cyclotomic graphs and perfect $e$-codes in them; they are Cayley graphs on the additive group of $\ZZZ[\zeta_m]/A$ with respect to $\{\pm (\zeta_m^i + A): 0 \le i \le m-1\}$ and $\{\pm (\zeta_m^i + A): 0 \le i \le \phi(m)-1\}$, respectively, where $\zeta_m$ ($m \ge 2$) is an $m$th primitive root of unity, $A$ is a nonzero ideal of $\ZZZ[\zeta_m]$, and $\phi$ is Euler's function. A family of circulant cyclotomic graphs of degree $2p$, $p$ being an odd prime, was identified in \cite[Section 5]{Z15} (see also \cite[Section 3]{TZ3}) as certain Cayley graphs associated with finite Frobenius groups. In the special case when $p=3$, a complete characterization of subgroup perfect $e$-codes in such circulant cyclotomic graphs was given in \cite[Theorem 4.7]{Z15} in conjunction with \cite[Theorem 24]{MBG}.

\subsection{Periodic sets}\label{sec:periodic}

We will only consider Cayley graphs of abelian groups. We will use additive notation for abelian groups and denote their identity elements by $0$. Thus, for an abelian group $G$, a subset $S$ of $G \setminus \{0\}$ is inverse-closed if $S=-S$, where $-S = \{-x: x \in S\}$, and for this connection set $S$, $x, y \in G$ are adjacent in $\cay(G, S)$ if and only if $x-y\in S$. Write
$$
S_0 = S\cup \{0\}
$$
and call it the \emph{extended connection set} of $\cay(G, S)$. It is readily seen that $\cay(G, S)$ is a regular graph with degree $|S|$, and $\cay(G, S)$ is connected if and only if $\langle S \rangle = G$, where for any subset $X \subseteq G$, $\langle X \rangle$ denotes the subgroup of $G$ generated by $X$. It can be verified that a subset $D$ of $G$ is a perfect code in $\cay(G, S)$ if and only if $D+g$ is a perfect code in $\cay(G, S)$ for every $g\in G$. Thus, in studying perfect codes in Cayley graphs of abelian groups, we may assume without loss of generality that $0$ is in the perfect code under consideration. Given a subgroup $H$ of $G$, let $\pi_H: G\to G/H, g \mapsto g+H$ denote the canonical projection. Slightly abusing notation, for $X \subseteq G$ we use $X/H$ to denote the image of $X$ under $\pi_H$. Any non-zero element $a \in G$ satisfying $X + a = X$ is called a \emph{period} of $X$. The periods of $X$ together with $0$ form a subgroup of $G$, denoted by $G_X$, which is called \cite{SS09} the \emph{subgroup of periods} of $X$. Note that $X$ is the union of some cosets of $G_X$ in $G$. If $G_X$ is nontrivial, then $X$ is \emph{periodic} (or \emph{$G_X$-periodic} in order to emphasize the subgroup $G_X$); otherwise, $X$ is \emph{aperiodic}.

\subsection{Main results}\label{sec:results}

As one can see above, even for circulant graphs (which are simplest Cayley graphs in a sense) we are far from a good understanding of perfect codes in them. In this paper we focus on existence, construction and enumeration of perfect codes in connected circulant graphs $\cay(\ZZZ_n,S)$ with $|S_0|$ a prime power. The following concept is vital for establishing our main results.

\begin{defn}
\label{def:pyra}
Let $G$ be an abelian group. A subset $X$ of $G$ with $0 \in X$ is called a \emph{pyramidal set} of $G$ if there exists a subgroup series
\begin{equation}
\label{eq:pyra}
H_0 < H_1 < \cdots < H_{2t-1} \le H_{2t} = G
\end{equation}
of $G$ with $t \ge 1$ such that either (a) or (b) below holds:
\begin{enumerate}[(a)]
\item $H_0 = \{0\}$, $X$ is aperiodic, and
\begin{enumerate}
\item[(T1)] $(X/H_{2i}-X/H_{2i})\cap (H_{2i+1}/H_{2i})=\{0\}$ for $0 \le i \le t-1$;
\item[(T2)] $X/H_{2i-1}$ is $(H_{2i}/H_{2i-1})$-periodic for $1 \le i \le t$;
\item[(T3)] $|X/H_{2t-1}| = |G/H_{2t-1}|$ and $|X/H_{2i}| = |X/H_{2i+1}|$ for $0 \le i  \le t-1$.
\end{enumerate}
\item $H_0 \ne \{0\}$, $X$ is $H_0$-periodic, and $X/H_0$ satisfies (T1), (T2) and (T3) in $G/H_0$ with respect to the subgroup series
$$
\{0\} = H_0/H_0 < H_1/H_0 < \cdots < H_{2t-1}/H_0 \le H_{2t}/H_0 = G/H_{0}
$$
of $G/H_0$.
\end{enumerate}
In both aperiodic and periodic cases, we call \eqref{eq:pyra} an \emph{admissible subgroup series} of $G$ with \emph{length} $2t$ associated with $X$.
\end{defn}

A pyramidal set is not necessarily an extended connection set. For example, $X=\{0,1,14,15,16,29,30,31,44\}$ is a pyramidal set of $\ZZZ_{315}$ with $H_0 < H_1 < H_2 < H_3 < H_4 = \ZZZ_{315}$ as an associated admissible subgroup series, where $H_0=\{0\}$, $H_1=\langle 45\rangle$, $H_2=\langle15\rangle$, $H_3=\langle 3\rangle$ and $H_4=\ZZZ_{315}$, but $X$ does not satisfy $-X = X$.

Condition (T2) above determines a pyramidal structure which gives the name to the set, as illustrated in Figure \ref{fig:pyramidal}. Condition (T3) is a rule for determining the length $2t$ of series \eqref{eq:pyra}. It may happen that there are multiple admissible subgroup series of $G$ associated with the same pyramidal set of $G$ (see Example \ref{ex:ex3.4}). However, as will be seen in Construction \ref{const:1} and Lemma \ref{lem:longest}, for any pyramidal set $X$ of the cyclic group $\ZZZ_n$, there is a unique longest admissible subgroup series of $\ZZZ_n$ associated with $X$.

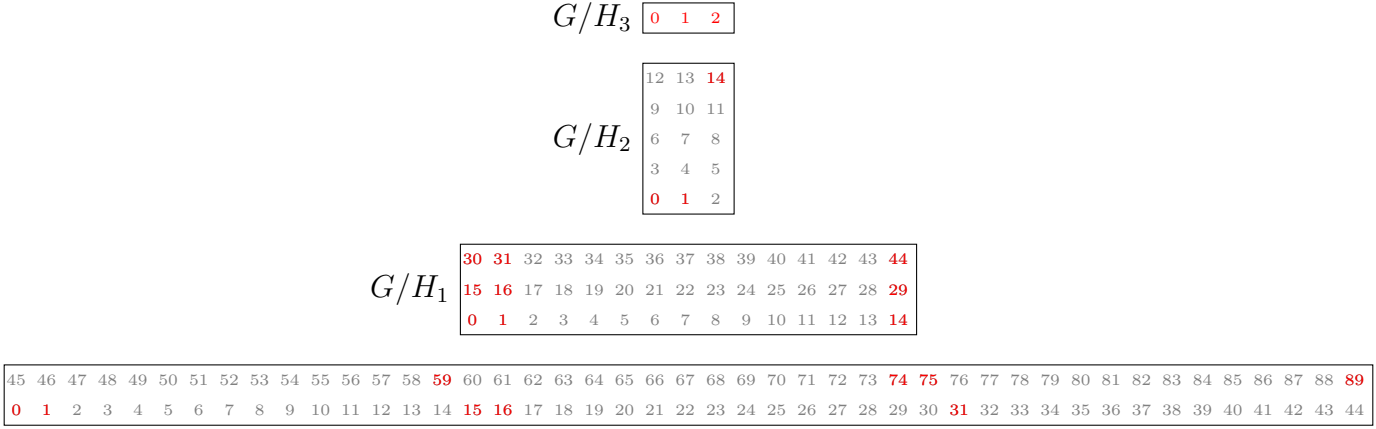
\begin{figure}
\begin{center}
\begin{tikzpicture}[scale=0.4]
\draw (-22.5,0) rectangle (22.5,2);
\foreach \i in {0,...,44}
{
\node[gray] at (\i+0.4-22.5,0.5) {\tiny\i};
}
\foreach \i in {45,...,89}
{
\node[gray] at (\i+0.4-45-22.5,1.5) {\tiny\i};
}
\node[red] at (0.4-22.5,0.5) {\tiny 0};
\node[red] at (0.4-22.5+1,0.5) {\tiny 1};
\node[red] at (0.4-22.5+15,0.5) {\tiny 15};
\node[red] at (0.4-22.5+16,0.5) {\tiny 16};
\node[red] at (0.4-22.5+31,0.5) {\tiny 31};
\node[red] at (0.4-22.5+14,1.5) {\tiny 59};
\node[red] at (0.4-22.5+29,1.5) {\tiny 74};
\node[red] at (0.4-22.5+30,1.5) {\tiny 75};
\node[red] at (0.4-22.5+44,1.5) {\tiny 89};

\draw (-7.5,3) rectangle (7.5,6);
\foreach \i in {0,...,14}
{
\node[gray] at (\i+0.4-7.5,3.5) {\tiny \i};
}
\foreach \i in {15,...,29}
{
\node[gray] at (\i+0.4-7.5-15,4.5) {\tiny \i};
}
\foreach \i in {30,...,44}
{
\node[gray] at (\i+0.4-7.5-30,5.5) {\tiny \i};
}
\node[red] at (0.4-7.5,3.5) {\tiny 0};
\node[red] at (0.4-7.5,4.5) {\tiny 15};
\node[red] at (0.4-7.5,5.5) {\tiny 30};

\node[red] at (0.4-6.5,3.5) {\tiny 1};
\node[red] at (0.4-6.5,4.5) {\tiny 16};
\node[red] at (0.4-6.5,5.5) {\tiny 31};

\node[red] at (0.4+6.5,3.5) {\tiny 14};
\node[red] at (0.4+6.5,4.5) {\tiny 29};
\node[red] at (0.4+6.5,5.5) {\tiny 44};

\node[left] at (-7.5,4.5) {$G/H_1$};
\draw (-1.5, 7) rectangle (1.5,12);
\foreach \i in {0,1,2}
\node[gray] at (\i+0.4-1.5, 7.5) {\tiny \i};
\foreach \i in {3,4,5}
\node[gray] at (\i+0.4-1.5-3, 8.5) {\tiny \i};
\foreach \i in {6,7,8}
\node[gray] at (\i+0.4-1.5-6, 9.5) {\tiny \i};
\foreach \i in {9,10,11}
\node[gray] at (\i+0.4-1.5-9, 10.5) {\tiny \i};
\foreach \i in {12,13,14}
\node[gray] at (\i+0.4-1.5-12, 11.5) {\tiny \i};

\node[red] at (0.4-1.5,7.5) {\tiny 0};
\node[red] at (0.4-1.5+1,7.5) {\tiny 1};
\node[red] at (0.4-1.5+2,11.5) {\tiny 14};

\node[left] at (-1.5,9.5) {$G/H_2$};
\draw (-1.5,13) rectangle (1.5,14);
\foreach \i in {0,1,2}
\node[red] at (\i+0.4-1.5,13.5) {\tiny \i};
\node[left] at (-1.5,13.5) {$G/H_3$};
\end{tikzpicture}
\end{center}
\caption{In $G=\ZZZ_{90}$, $S_0=\{0,1,15,16,31,59,74,75,89\}$ is a pyramidal set with the longest admissible subgroup series $H_1=\langle 45\rangle, H_2=\langle 15\rangle, H_3=\langle 3 \rangle$. For $i = 1, 2, 3$, the members of $G/H_i$ in red colour form $S_0/H_i$. This diagram looks like a pyramid, hence the eponym.}
\label{fig:pyramidal}
\end{figure}

Our first main result, stated below, gives a necessary and sufficient condition for a connected circulant graph of degree $p^{\ell}-1$ to admit perfect codes. Moreover, it tells us when a perfect code in such a graph with an aperiodic extended connection set can be a subgroup perfect code.

\begin{theorem}\label{th:thplp}
Let $n \ge 3$ be an integer, and let $\cay(\ZZZ_n, S)$ be a connected non-complete circulant graph with order $n$ and degree $p^{\ell}-1$, where $p$ is a prime and $\ell \ge 1$ is an integer. Then $\cay(\ZZZ_n, S)$ admits a perfect code if and only if $p^{\ell}$ divides $n$ and $S_0$ is a pyramidal set of $\ZZZ_n$.

Moreover, if $p^{\ell}$ divides $n$, $S_0$ is pyramidal and aperiodic, and the longest admissible subgroup series associated with $S_0$ has length $2t$, then the following hold:
\begin{itemize}
\item[(a)] if $(S_0-S_0)\cap \langle p^\ell \rangle=\{0\}$, then the subgroup $\langle p^\ell \rangle$ of $\ZZZ_n$ is a perfect code in $\Gamma$ containing $0$; furthermore, when $t=1$ this is the only perfect code in $\Gamma$ containing $0$, whilst when $t > 1$ there are perfect codes in $\Gamma$ containing $0$ which are not subgroups of $\ZZZ_n$;
\item[(b)] if $(S_0-S_0)\cap \langle p^\ell \rangle\neq\{0\}$, then there are perfect codes in $\Gamma$ containing $0$, but none of them is the subgroup $\langle p^\ell \rangle$ of $\ZZZ_n$.
\end{itemize}
\end{theorem}

As will be seen in Lemma \ref{lem:iota}, if $S_0$ is pyramidal and aperiodic, then every admissible subgroup series $\iota: H_0<H_1<\cdots<H_{2t-1} \le H_{2t} = \ZZZ_n$ associated with $S_0$ gives rise to two sequences of positive integers (see \eqref{eq:lm}), namely $\mathbf{l}_{\iota} = (\ell_1, \ell_2, \ldots, \ell_{t})$ and $\mathbf{m}_{\iota} = (m_1, m_2, \ldots, m_{t})$, such that $\sum_{i=1}^t \ell_i =\ell$, $\prod_{i=1}^t m_i$ divides $m$, $|H_{2i}/H_{2i-1}|=p^{\ell_i}$ and $|H_{2i-1}/H_{2i-2}|=m_i$ for $1 \leq i \leq t$, and $\cay(\ZZZ_n/H_{2t-1}, (S_0/H_{2t-1})\setminus \{H_{2t-1}\}) \cong K_{p^r}$, where $r = \ell_t$. We will see in Construction \ref{const:2} that $\Ga$ can be reduced to $K_{p^r}$ through a sequence of projections. Conversely, for any sequences $\mathbf{l} = (\ell_1, \ell_2, \ldots, \ell_t), \mathbf{m} = (m_1, m_2, \ldots, m_{t})$ of positive integers such that $\sum_{i=1}^t \ell_i =\ell$ and $\prod_{i=1}^{t}m_i = m$, we will construct a family of connected circulant graphs with order $n$ and degree $p^\ell - 1$ through $2t+1$ lifts of $K_{p^r}$ (see Construction \ref{const:3}), where $r = \ell_t$. Denote this family of circulant graphs by $\GG_{\mathbf{l}, \mathbf{m}}(n, p^\ell)$. Our second main result, Theorem \ref{th:th2.2.1} below, asserts that this construction produces all circulant graphs of degree $p^\ell - 1$ which admit perfect codes.

\begin{theorem}
\label{th:th2.2.1}
Let $n=p^\ell m$, where $p$ is a prime and $\ell\geq 1, m\geq 2$ are integers. Then a connected circulant graph with order $n$ and degree $p^\ell - 1$ admits a perfect code if and only if it is isomorphic to a graph in $\GG_{\mathbf{l}, \mathbf{m}}(n, p^\ell)$ for some sequences $\mathbf{l} = (\ell_1, \ell_2, \ldots, \ell_t), \mathbf{m} = (m_1, m_2, \ldots, m_t)$ of positive integers satisfying $\sum_{i=1}^t \ell_i = \ell$ and $\prod_{i=1}^{t}m_i = m$.
\end{theorem}

It is natural to ask how many distinct perfect codes containing $0$ a connected circulant graph satisfying the condition in Theorem \ref{th:thplp} may have. Our third main result, Theorem \ref{cor:cor2.2} below, provides a lower bound on this number. This bound is given in terms of two sequences of integers obtained from the longest admissible subgroup series of $\ZZZ_n$ associated with the pyramidal set $S_0$, namely $\mathbf{h}(S_0)$ and $\mathbf{k}(S_0)$ (see \eqref{eq:hk}) to be given in Construction \ref{const:1}.

\begin{theorem}
\label{cor:cor2.2}
Let $n=p^\ell m$, and let $\cay(\ZZZ_n,S)$ be a connected circulant graph with order $n$ and degree $p^\ell-1$, where $p$ is a prime and $\ell\geq 1, m\geq 2$ are integers. Assume that $S_0$ is pyramidal with associated sequences $\mathbf{h}(S_0) = (h_0, h_1, \ldots, h_{t-1}, h_t)$ and $\mathbf{k}(S_0) = (k_0, k_1, \ldots, k_{t-1}, k_t)$. Then the number $c(S)$ of distinct perfect codes containing $0$ in $\cay(\ZZZ_n,S)$ satisfies
$$
c(S)\ge p^{\left(\sum_{i=0}^{t-1}(h_{i-1}-h_{i})k_{i}\right)+h_{t-1}-\ell},
$$
where $h_{-1} = \ell$.
\end{theorem}

\subsection{Organisation of the paper}

The rest of the paper is organized as follows. As preparations for the proof of our main results, in the next section we present several known results and prove two lemmas on admissible subgroup series associated with a pyramidal set. Theorem \ref{th:thplp} will be proved in Section \ref{sec:char}. In Section \ref{sec:const}, we give two constructions (Constructions \ref{const:2} and \ref{const:3}), hence defining the family $\GG_{\mathbf{l}, \mathbf{m}}(n, p^\ell)$, prove several lemmas around them, and use these lemmas to prove Theorem \ref{th:th2.2.1}. In Section \ref{sec:enum}, we first construct a specific admissible subgroup series associated with a pyramidal extended connection set $S_0$ of $\ZZZ_n$ (where $n = p^l m$) with $|S_0| = p^l$ together with the integer sequences $\mathbf{h}(S_0)$ and $\mathbf{k}(S_0)$ (Construction \ref{const:1}), and then prove that this series is the unique longest admissible subgroup series of $\ZZZ_n$ associated with $S_0$ (Lemma \ref{lem:longest}). Based on these we give a method for constructing various perfect codes in $\cay(\ZZZ_n, S)$ containing $0$ (Construction \ref{const:DD}) and use this construction to prove Theorem \ref{cor:cor2.2}. The paper concludes with some final comments and remarks in Section \ref{sec:final}.


\section{Preliminaries}

In our proof of Theorem \ref{th:thplp} we will use the following result.

\begin{theorem}
\label{le:leobr07}
(\cite[Theorem 3.1]{deng14}) Let $n \ge 3$ be an integer, and let $S$ be a connection set of $\ZZZ_n$ such that $|S_0|$ divides $n$. Then $\cay(\ZZZ_n,S)$ admits a perfect code $H$ which is a subgroup of $\ZZZ_n$ if and only if $(S_0-S_0)\cap H = \{0\}$. Moreover, under this condition $H$ must be the subgroup $\langle |S_0| \rangle$ of $\ZZZ_n$ generated by $|S_0|$.
\end{theorem}

In our proof of Theorems \ref{th:thplp} and \ref{th:th2.2.1} we will use the following result.

\begin{theorem}
\label{le:ledeng17p}
(\cite[Theorem 3.1]{deng17} and \cite[Theorem 1.1]{fhz17})
Let $n \ge 3$ be an integer, and let $\cay(\ZZZ_n,S)$ be a connected non-complete circulant graph with order $n$ and degree $|S| = p-1$, where $p$ is a prime. Then $\cay(\ZZZ_n,S)$ admits a perfect code if and only if $p$ divides $n$ and $(S_0-S_0)\cap H = \{0\}$, where $H = \langle p\rangle$ is the subgroup of $\ZZZ_n$ generated by $|S_0|$.
\end{theorem}

Let $G$ be an abelian group. For two subsets $A, B$ of $G$, their sumset is defined as $A+ B=\{a+b: a\in A\text{ and } b\in B\}$. Set $-B = \{-b: b \in B\}$ and $A - B = A + (-B)$. If the elements in $A+B$ are pairwise distinct, then $A+B$ is a \emph{direct sum} and we write $A\oplus B$ to indicate this. If $A\oplus B = G$, then we say that $G = A\oplus B$ is a \emph{factorization} of $G$ into $A$ and $B$; if in addition $0 \in A \cap B$, then $(A, B)$ is called a \emph{tiling} of $G$. It can be verified (see, for example, \cite{deng14}) that a subset $D$ of $G$ is a perfect code in a Cayley graph $\cay(G,S)$ of $G$ if and only if $G = S_0\oplus D$. For a subgroup $H$ of $G$, a \emph{transversal} of $H$ in $G$ is a subset $X$ of $G$ which contains precisely one element from each coset of $H$ in $G$. It can be verified that $H$ is a perfect code in $\cay (G, S)$ if $S_0$ is a transversal of $H$ in $G$.

\begin{lemma}
\label{le:ledirectsum}
\cite[Lemma 2.2]{SS09}
Let $G$ be an abelian group, and let $A$ and $B$ be non-empty subsets of $G$. The
following statements are all equivalent:
\begin{itemize}
\item[(a)] $G=A\oplus B$;
\item[(b)] $G=A+B$ and $|G|=|A||B|$;
\item[(c)] $|G|=|A||B|$ and $(A-A)\cap(B-B)\subseteq \{0\}$;
\item[(d)] $G=A+B$ and $(A-A)\cap(B-B)\subseteq \{0\}$.
\end{itemize}
\end{lemma}

\begin{lemma}
\label{le:ledeng1424}
\cite[Lemma 2.4]{deng14}
Let $A$ and $B$ be subsets of an abelian group $G$. Assume that $A$ is periodic with subgroup of periods $H$. Then $G=A\oplus B$ if and only if $G/H=(A/H)\oplus (B/H)$ and $(B-B)\cap H=\{0\}$.
\end{lemma}

\begin{lemma}
\label{le:lesandsbook}
\cite[Theorem~3.19]{SS09}
If an abelian group $G$ can be factorized as $G = H \oplus K = A\oplus B$, where $H, K$ are subgroups of $G$ with relatively prime orders and $A, B$ are subsets of $G$ with $|A|=|H|$ and $|B|=|K|$, then $G=H\oplus B=A\oplus K$.
\end{lemma}

\begin{lemma}\label{le:lesdtgd}
Let $H$ be a subgroup of $\ZZZ_n$, and let $A$ and $B$ be non-empty subsets of $\ZZZ_n$ with $|A/H|=|A|$. If $\ZZZ_n/H=(A/H) \oplus (B/H)$, then $\ZZZ_n=A\oplus (B+H)$.
\end{lemma}

\begin{proof}
Since $\ZZZ_n/H=(A/H) \oplus (B/H)$ and $|A/H|=|A|$, we have $n/|H|=|\ZZZ_n/H|=|A/H||B/H|=|A||B/H|$. Since $H$ is a subgroup of $\ZZZ_n$, it is of the form $H=\langle y \rangle$ for some integer $y$ between $1$ and $n-1$. Define $D=\{d<y:d+H\in B/H\}$. Since $\{0, 1, \ldots, y-1\}$ is a transversal of $H$ in $\ZZZ_n$, we have $|D|=|B/H|$.

Define $\phi:D+H\rightarrow B+H$ such that $\phi((d, h))=d+h$ for $d \in D$ and $h\in H$. If $\phi((d_1,h_1))=\phi((d_2,h_2))$ for $(d_1, h_1), (d_2, h_2) \in D+H$, then $d_1+h_1=d_2+h_2$, which implies $d_1=d_2$ and so $h_1=h_2$ as $0\leq d_1,d_2 < y$. Hence $\phi$ is injective. Also, for any $x \in B+H$, there exists an integer $z$ with $0 \leq z < y$ such that $x=z+h=\phi((z,h))$. So $\phi$ is surjective as well. Therefore, $\phi$ is a bijection and hence $|D||H|=|B+H|$. Thus, $n=|A||B/H||H|=|A||D||H|=|A||B+H|$. Since $\ZZZ_n/H=(A/H) \oplus (B/H)$, we then have $\ZZZ_n=A+(B+H)$. This together with Lemma \ref{le:ledirectsum} yields $\ZZZ_n=A\oplus(B+H)$.
\end{proof}

If $G$ is a cyclic group and $G=A\oplus B$ is a factorization of $G$ with $\gcd (|A|,|B|)=1$, then the conclusion in Lemma \ref{le:lesandsbook} follows automatically because in this case there is a unique pair of subgroups $H, K$ of $G$ with $|H|=|A|$ and $|K|=|B|$. For cyclic groups, we have the following more specific result.

\begin{lemma}
\label{thm:primepower}
\cite[Theorem 4.4]{SS09}
Let $G$ be a cyclic group and $G=A\oplus B$ a factorization of $G$. If $|A|=p^l$ with $p$ a prime, then either $A$ or $B$ is periodic.	
\end{lemma}
	
Let $\Gamma$ and $\Si$ be graphs. The {\em wreath product} of $\Gamma$ and $\Si$ is the graph $\Gamma \otimes \Si$ with vertex set $V(\Gamma)\times V(\Si)$ such that $(u, v), (x,y)$ are adjacent if and only if  either $u, x$ are adjacent in $\Ga$ or $u=x$ and $v, y$ are adjacent in $\Si$.

\begin{lemma}
\label{le:lekum}
\cite[Proposition 3.2]{kum13}
Let $\Gamma$ be a graph without isolated vertices and $\Si$ a graph with at least one vertex. Then $\Gamma \otimes \Si$ admits a perfect code if and only if $\Gamma$ admits a perfect code and $\Si$ is a complete graph.
\end{lemma}

\begin{lemma}
\label{le:wreath}
Let $\cay(\ZZZ_{n}, S)$ be a circulant graph such that $S_0$ is periodic with subgroup of periods $H$. Then
\begin{equation}
\label{eq:eq1.1}
\cay(\ZZZ_n, S) \cong \cay(\ph(\ZZZ_n), \ph(S)\setminus \{H\}) \otimes \cay(H, H\setminus \{0\}).
\end{equation}
\end{lemma}

\begin{proof}
Let $X$ be a transversal of $H$ in $G$. Then the map $f:G\to (G/H) \times H$ defined by $f(g)=(\ph(x),h)$, where $x \in X$ and $h\in H$ are the unique pair of elements of $G$ such that $g=x+h$, is a graph isomorphism from the left-hand side to the right-hand side of \eqref{eq:eq1.1}.
\end{proof}

Note that $\cay(H, H\setminus \{0\})$ is the complete graph on $H$. The following lemma says that the converse of the statement in Lemma \ref{le:wreath} is also true.

\begin{lemma}
\label{le:legraphiso}
\cite[Theorem 28]{bh90}
A circulant graph $\cay(\ZZZ_{mn}, S)$ with order $mn$ is isomorphic to the wreath product of a circulant graph with order $n$ and the complete graph $K_m$ if and only if $S_0$ is $\langle n\rangle$-periodic.
\end{lemma}

It follows from Lemmas \ref{le:lekum} and \ref{le:wreath} that, if $S_0$ is $H$-periodic, then every perfect code $C$ in $\Gamma = \cay(\ZZZ_n,S)$ can be obtained from a perfect code $D = \{d_1, \ldots, d_t\}$ in $\Si = \cay(\ph(\ZZZ_n), \ph(S)\setminus \{H\})$ and a collection of (not necessarily distinct) elements $h_1, \ldots, h_t \in H$ in such a way that $C=\{(d_1, h_1), \ldots, (d_t, h_t)\}$, and vice versa. Thus we can obtain $|H|^{t-1}$ distinct perfect codes containing $0$ in $\Gamma$ from a given perfect code $D$ in $\Sigma$. Moreover, for distinct perfect codes $D = \{d_1, \ldots, d_t\}$ and $D' = \{d'_1, \ldots, d'_t\}$ in $\Gamma$, we have $d'_i \notin D$ for some $1 \le i \le t$ and hence $\{(d_1, h_1), \ldots, (d_t, h_t)\} \ne \{(d'_1, h'_1), \ldots, (d'_t, h'_t)\}$ for any $h_1, \ldots, h_t \in H$ and $h'_1, \ldots, h'_t \in H$. Therefore, if there are exactly $c(\Si)$ distinct perfect codes containing $0$ in $\Si$, then there are exactly $c(\Si)\cdot |H|^{t-1}$ distinct perfect codes containing $0$ in $\Gamma$. Moreover, $\pi_{H} (S_0)$ is aperiodic as $H$ is the subgroup of periods of $S_0$.

\begin{lemma}
\label{le:leqas}
Let $G$ be an abelian group. Let $X$ be a pyramidal set of $G$ with
\begin{equation*}
\iota: H_0<H_1<\cdots<H_{2t-1} \le H_{2t} = \ZZZ_n
\end{equation*}
as an associated admissible subgroup series. Then the following statements hold:
\begin{enumerate}[\rm (a)]
\item
if $H_0=\{0\}$ and $X$ is aperiodic, then for each $0 \le i \le t-1$, $X/H_{2i}$ is a pyramidal set of $G/H_{2i}$ with
\begin{equation}
\label{eq:iota1}
\iota': H_{2i}/H_{2i}<H_{2i+1}/H_{2i}<\cdots<H_{2t-1}/H_{2i} \le H_{2t}/H_{2i} = \ZZZ_n/H_{2i}
\end{equation}
as an associated admissible subgroup series;
\item
if $H_0\neq\{0\}$ and $X$ is $H_0$-periodic, then for each $0 \le i \le t-2$, $X/H_{2i+1}$ is a pyramidal set of $G/H_{2i+1}$ with
\begin{equation}
\label{eq:iota2}
\iota'': H_{2i+2}/H_{2i+1}<H_{2i+3}/H_{2i+1}<\cdots<H_{2t-1}/H_{2i+1} \le H_{2t}/H_{2i+1} = \ZZZ_n/H_{2i+1}
\end{equation}
as an associated admissible subgroup series.
\end{enumerate}
\end{lemma}

\begin{proof}
Since $X$ is a pyramidal set of $G$, we have $0 \in X$.

(a) Since $H_0=\{0\}$ and $X$ is an aperiodic pyramidal set of $G$ with $\iota$ as an associated admissible subgroup series, by (T2), for $0 \le i \le t-1$, $X/H_{2i-1}$ is $(H_{2i}/H_{2i-1})$-periodic and so $X/H_{2i}$ is aperiodic. By (T1), we have $\left((X/H_{2i})/(H_{2j}/H_{2i})-(X/H_{2i})/(H_{2j}/H_{2i})\right)\cap\left((H_{2j+1}/H_{2i})/(H_{2j}/H_{2i})\right) = (X/H_{2j}-X/H_{2j})\cap(H_{2j+1}/H_{2j}) = \{0\}$ for $i\leq j\le t-1$. By (T2), for $i+1\le j\le t$, $(X/H_{2i})/(H_{2j-1}/H_{2i}) = X/H_{2j-1}$ is $(H_{2j}/H_{2j-1})$-periodic, so $(X/H_{2i})/(H_{2j-1}/H_{2i})$ is $((H_{2j}/H_{2i})/(H_{2j-1}/H_{2i}))$-periodic. Moreover, for $i+1\le j\le t-1$, we have $|(X/H_{2i})/(H_{2j}/H_{2i})|$ $=|X/H_{2j}|=|X/H_{2j+1}|=|(X/H_{2i})/(H_{2j+1}/H_{2i})|$ and $|(X/H_{2i})/(H_{2t-1}/H_{2i})|=|X/H_{2t-1}|=|G/H_{2t-1}|=|(G/H_{2i})/(H_{2t-1}/H_{2i})|$. Thus $\iota'$ satisfies (T1)--(T3) and so for any $0 \le i \le t-1$, $X/H_{2i}$ is a pyramidal set of $G/H_{2i}$ with $\iota'$ as an associated admissible subgroup series.

(b) Since $H_0\neq\{0\}$ and $X$ is an $H_0$-periodic pyramidal set of $G$ with $\iota$ as an associated admissible subgroup series, by (T2), for $0 \le i \le t-2$, $X/H_{2i+1}$ is $(H_{2i+2}/H_{2i+1})$-periodic and $H_{2i+2}/H_{2i+1}\neq\{0\}$. Note that $(H_{k}/H_{2i+1})/(H_{2i+2}/H_{2i+1})=H_{k}/H_{2i+2}$ for $2i+1 \le k \le 2t$. Similarly to part (a), using the definition of a pyramidal set, we can show that $X/H_{2i+1}$ is a pyramidal set of $G/H_{2i+1}$ with $\iota''$ as an associated admissible subgroup series.
\end{proof}

\begin{lemma}
\label{lem:iota}
Let $n=p^\ell m$, where $p$ is a prime and $\ell\geq 1, m\geq 2$ are integers. Let $\cay(\ZZZ_n,S)$ be a connected circulant graph with order $n$ and degree $p^\ell - 1$ such that $S_0$ is pyramidal. Then every admissible subgroup series
\begin{equation}
\label{eq:iota}
\iota: H_0<H_1<\cdots<H_{2t-1} \le H_{2t} = \ZZZ_n
\end{equation}
of $\ZZZ_n$ associated with $S_0$ (where $t \ge 1$) determines two sequences of positive integers
\begin{equation}
\label{eq:lm}
\mathbf{l}_{\iota} = (\ell_1, \ell_2, \ldots, \ell_{t}),\; \mathbf{m}_{\iota} = (m_1, m_2, \ldots, m_{t})
\end{equation}
with the following properties:
\begin{enumerate}[\rm (a)]
\item
$\prod_{i=1}^t m_i$ divides $m$;
\item
$\sum_{i=1}^t \ell_i = \ell$ if $S_0$ is aperiodic, and $\sum_{i=1}^t \ell_i < \ell$ if $S_0$ is periodic;
\item
for $1 \leq i \leq t$, we have
\begin{equation}
\label{eq:H}
|H_{2i}/H_{2i-1}|=p^{\ell_i},\;\, |H_{2i-1}/H_{2i-2}|=m_i;
\end{equation}
\item
$\cay(\ZZZ_n/H_{2t-1}, (S_0/H_{2t-1})\setminus \{H_{2t-1}\}) \cong K_{p^{\ell_t}}$.
\end{enumerate}
\end{lemma}

\begin{proof}
We prove the result by induction on $t$. Assume $t=1$ first. Then $\iota: H_0<H_1\le H_{2} = \ZZZ_n$. Since $\iota$ is an admissible subgroup series of $\ZZZ_n$ associated with $S_0$, either (i) $H_0=\{0\}$ and $S_0$ is aperiodic, or (ii) $H_0\neq \{0\}$ and $S_0$ is $H_0$-periodic. In the former case, by (T3), we have $|S_0/H_1|=|\ZZZ_n/H_1|$ and $|S_0/H_0|=|S_0/H_1|$. Since $H_0=\{0\}$, it follows that $p^\ell=|S_0|=|S_0/H_0|=|S_0/H_1|=|\ZZZ_n/H_1|$, and so $H_1=\langle p^\ell\rangle$ and $|H_1|=m$. Set $\ell_1=\ell$ and $m_1=m$. Then $\mathbf{l}_{\iota} = (\ell)$ and $\mathbf{m}_{\iota} = (m)$. By (T2) and (T3), we obtain that $S_0/H_1$ is $(\ZZZ_n/H_1)$-periodic, which implies that $\ZZZ_n/H_1$ is a subset of $S_0/H_1$. Since $S_0\subseteq \ZZZ_n$, it follows that $S_0/H_1=\ZZZ_n/H_1$. Therefore, $S_0/H_1$ is a group with order $p^\ell$ and consequently $\cay(\ZZZ_n/H_1,(S_0/H_1)\setminus \{H_1\})\cong K_{p^\ell}$ as desired.

Now let us consider the case where $H_0\neq \{0\}$ and $S_0$ is $H_0$-periodic. In this case there is an  integer $\ell_1 < \ell$ such that $H_0=\langle p^{\ell_1} m\rangle$ and $S_0/H_0$ is a pyramidal set of $\ZZZ_n$ with $\{0\}=H_0/H_0<H_1/H_0\leq H_2/H_0=\ZZZ_n/H_0\cong \ZZZ_{p^{\ell_1} m}$ as an associated admissible subgroup series. Set $L_0=H_0/H_0$, $L_1=H_1/H_0$ and $L_2=\ZZZ_n/H_0$. Since $S_0/H_0$ is aperiodic and $L_0=\{0\}$, we can use what we have proved above in case (i) to obtain $\mathbf{l}_{\iota} = (\ell_1)$ and $\mathbf{m}_{\iota} = (m)$. We have $|H_2/H_1|=|L_2/L_1|=p^{\ell_1}$ and $|H_1/H_0|=|L_1/L_0|=m$. Similarly to case (i), we can obtain that $\cay(\ZZZ_n/L_1, S_{0}/L_1\setminus\{L_1\}) \cong K_{p^{\ell_1}}$.

Suppose that for some integer $k > 1$ the result is true for all $t < k$. We aim to prove that the result is also true for $t=k$. We distinguish between the following two cases.

\smallskip
\textsf{Case 1.} $H_0=\{0\}$ and $S_0$ is aperiodic.
\smallskip

By (T2), $S_0/H_1$ is $(H_2/H_1)$-periodic. Since $H_0=\{0\}$, by (T3), $|S_0/H_1|=|S_0/H_{0}|=|S_0|=p^\ell$. Thus there is an integer $a<\ell$ such that $|H_2/H_1|=p^a$. Since $S_0/H_1$ is $(H_2/H_1)$-periodic, by Theorem \ref{th:thplp}, Lemma \ref{le:wreath} and \cite[Proposition 3.2]{kum13}, $\cay(\ZZZ_n/H_1, S_{0}/H_1\setminus\{H_1\})$ also admits perfect codes. Let $D$ be a perfect code in $\cay(\ZZZ_n/H_1,S_0/H_1\setminus\{H_1\})$. By Lemma \ref{le:lesdtgd} and (T1), we have $\ZZZ_n=S_0\oplus (D+H_1)$. Thus $|H_1|$ is a divisor of $m$. Set $|H_1|=m'$. Then $|H_1/H_0|=|H_1|=m'$, and so $H_1=\langle p^\ell m/m'\rangle$. This together with $|H_2/H_1|=p^a$ implies that $H_2=\langle p^{\ell-a}m/m'\rangle$.

Now by Lemma \ref{le:leqas}, $S_0/H_2$ is a pyramidal set of $\ZZZ_n/H_2$, and
$$
\iota': H_{2}/H_{2}<H_{3}/H_{2}<\cdots<H_{2k-1}/H_{2} \le H_{2k}/H_{2} = \ZZZ_n/H_{2}
$$
is an associated admissible subgroup series. Thus, by Theorem \ref{th:thplp}, $\cay(\ZZZ_n/H_2,S_0/H_2\setminus\{H_2\})$ admits perfect codes. By the induction hypothesis, we have two sequences of positive integers
$$
\mathbf{l'}_{\iota} = (a_1, a_2, \ldots, a_{k-1}),\; \mathbf{m'}_{\iota} = (b_1, b_2, \ldots, b_{k-1})
$$
with $\sum_{i=1}^k a_i = \ell-a$ and $\prod_{i=1}^k b_i$ dividing $m/m'$ such that
$$
|H_{2i}/H_{2i-1}|=p^{a_i},\;\, |H_{2i-1}/H_{2i-2}|=b_i
$$
for $2 \leq i \leq k$ and $\cay(\ZZZ_n/H_{2k-1}, (S_0/H_{2k-1})\setminus \{H_{2k-1}\}) \cong K_{p^r}$, where $r = a_{k-1}$. Set
$$
\mathbf{l}_{\iota} = (a,a_1, a_2, \ldots, a_{k-1}),\; \mathbf{m}_{\iota} = (m', b_1, b_2, \ldots, b_{k-1}).
$$
Then $\mathbf{l}_{\iota}$ and $\mathbf{m}_{\iota}$ satisfy \eqref{eq:H}. So we have proved that the result is true for $t=k$ when $H_0=\{0\}$ and $S_0$ is aperiodic.

\smallskip
\textsf{Case 2.} $H_0\neq\{0\}$ and $S_0$ is $H_0$-periodic.
\smallskip

In this case,
$$
\iota'': H_{0}/H_{0}<H_{1}/H_{0}<\cdots<H_{2k-1}/H_{0} \le H_{2k}/H_{0} = \ZZZ_n/H_{0}
$$
is an admissible subgroup series associated with $S_0$. Since $S_0$ is $H_0$-periodic, there is an integer $1\leq a<\ell$ such that $H_{0}=\langle p^{\ell-a}m\rangle$ and $|S_0/H_0|=p^{\ell-a}$. Thus, by what we proved in Case 1, we have two sequences of positive integers
$$
\mathbf{l}_{\iota} = (\ell_1, \ell_2, \ldots, \ell_{t}),\; \mathbf{m}_{\iota} = (m_1, m_2, \ldots, m_{t})
$$
with $\sum_{i=1}^t \ell_i = \ell-a$ and $\prod_{i=1}^t m_i$ dividing $m$ such that
$$
|H_{2i}/H_{2i-1}|=p^{\ell_i},\;\, |H_{2i-1}/H_{2i-2}|=m_i
$$
for $1 \leq i \leq t$ and $\cay(\ZZZ_n/H_{2k-1}, (S_0/H_{2k-1})\setminus \{H_{2k-1}\}) \cong K_{p^{\ell_t}}$. Since $\ell-a<\ell$, we have $\sum_{i=1}^t \ell_i < \ell$. Moreover, $\mathbf{l}_{\iota}$ and $\mathbf{m}_{\iota}$ satisfy \eqref{eq:H}. So we have proved that the result is true for $t=k$ when $H_0\neq\{0\}$ and $S_0$ is $H_0$-periodic.

By mathematical induction, we have completed the proof.
\end{proof}


\section{Proof of Theorem \ref{th:thplp}}
\label{sec:char}

\begin{Proof}~\ref{th:thplp}.
Throughout the proof we assume that $G = \ZZZ_n$ and $\Gamma =\cay (\ZZZ_n, S)$ is a connected non-complete circulant graph with order $n \ge 3$ and degree $|S| = p^{\ell}-1$, where $p$ is a prime and $\ell \ge 1, m \ge 2$ are integers. We prove the necessary and sufficient condition first.

\smallskip
\textsf{Necessity.}~Suppose that $\Gamma$ admits a perfect code, say, $D$, so that $G=S_0\oplus D$. Then by Lemma \ref{le:ledirectsum}, $|S_0| = p^{\ell}$ is a divisor of $n$. Write $n = p^{\ell}m$. We will prove by induction on $\ell$ that $S_0$ is a pyramidal set of $G$. The base case $\ell=1$ is ensured by Theorem \ref{le:ledeng17p}. In fact, if $S_0$ is periodic with size $p$, then it is the subgroup of $G$ with order $p$ and so cannot be a generating set of $G$ as $m \ge 2$. Thus, $S_0$ is aperiodic and pyramidal with subgroup sequence $\{0\}=H_0<H_1<H_2=G$, where $H_1=\langle p\rangle$.

Suppose $\ell>1$. By Theorem \ref{le:leobr07}, $\Gamma$ admits a perfect code $H$ which is a subgroup of $G$ if and only if $H=\langle p^{\ell}\rangle$ and $(S_0-S_0)\cap H=\{0\}$. In this case $S_0$ is a pyramidal set with admissible subgroup series $\{0\}=H_0<H_1<H_2=G$, where $H_1=\langle p^{\ell}\rangle$. It remains to consider the case where $(S_0-S_0)\cap \langle p^{\ell}\rangle\neq \{0\}$. We claim that $\gcd(p, m)\neq 1$ in this case. Suppose otherwise. Then by Lemma \ref{le:lesandsbook} there are subgroups $H, K$ of $G$ such that $|H|=|S_0|=p^{\ell}$, $|K|=|D|=m$, and $G=S_0\oplus D=H\oplus K=S_0\oplus K$. Since the sum $S_0\oplus K$ is direct,  the elements of $S_0$ are pairwise distinct modulo $K$, and $K$ is the cyclic group of $G$ generated by $p^{\ell}$. Therefore, $(S_0-S_0)\cap \langle p^{\ell}\rangle=\{0\}$, but this is a contradiction. Hence $\gcd(p, m)\neq 1$ as claimed.
	
\smallskip
\textsf{Case 1.} $S_0$ is periodic.
\smallskip

The subgroup of periods of $S_0$ must be of the form $H=\langle p^{\ell_1}m\rangle$ for some $\ell_1<\ell$. Hence $|S_0/H|=p^{\ell_1}$. By Theorem \ref{le:leobr07} and Lemma \ref{le:ledeng1424}, $G/H=(S_0/H) \oplus (D/H)$ and $(D-D)\cap H=\{0\}$.  So, by the induction hypothesis, $S_0/H$ is a pyramidal set of $G/H$. Therefore, $S_0$ is a pyramidal set of $G$ by definition.

\smallskip
\textsf{Case 2.} $S_0$ is aperiodic.
\smallskip

In this case, by Lemma \ref{thm:primepower}, $D$ must be periodic, say,  $D=D_1\oplus K$ for some subgroup $K$ with order $m_1$ dividing $m$. By Theorem \ref{le:leobr07}, $G/K = (S_0/K) \oplus (D/K)$ and $(S_0-S_0)\cap K=\{0\}$. The second condition implies that $|S_0/K|=p^{\ell}$. By a further application of Lemma \ref{thm:primepower}, we see that $S_0/K$ is periodic, say, with subgroup of periods $H/K$. Thus, by what we have proved in Case 1, $S_0/K$ is a pyramidal set of $G/K$. Hence, if $H/K$ is the subgroup of periods of $S_0/K$, where $H$ is a subgroup containing $K$, then $S_0/H$ is an aperiodic pyramidal set of $G/H$. Let $\{0\}<H_1/H<\cdots <H_{2t-1}/H< H_{2t}/H = G/H$ be an admissible subgroup series of $G/H$ associated with $S_0/H$. Then $\{0\}<K<H<H_1<\cdots <H_{2t-1}\leq H_{2t}=G$ is an admissible subgroup series of $G$ associated with $S_0$ and therefore $S_0$ is a pyramidal set of $G$.

\smallskip
\textsf{Sufficiency.}~Suppose that $p^{\ell}$ divides $n$ and $S_0$ is pyramidal. Then $n = p^{\ell}m$ for some integer $m \ge 2$. (Note that $m \ne 1$ for otherwise $\Ga$ would be a complete graph, which contradicts our assumption.) Let $H_0<H_1<\cdots <H_{2t-1} \leq H_{2t} = G$ be an admissible subgroup series of $G$ associated with $S_0$. To prove that $\Ga$ admits a perfect code, it suffices to prove the existence of a subset $D$ of $G$ such that $G = S_0\oplus D$.

Suppose first that $S_0$ is aperiodic. Then $H_0 = \{0\}$. If $t=1$, then, by (T1), $(S_0-S_0)\cap H_1=0 $. In particular, $|S_0/H_1|=|S_0|$ and, by (T3), $n=|S_0/H_1|\cdot |H_1|=|S_0|\cdot |H_1|$. Therefore, the sum $S_0\oplus H_1$ is direct and is equal to $G$, completing the proof in this case. Assume $t > 1$. By (T3), $|G/H_{2t-1}|=|S_0/H_{2t-1}|$. So $|(G/H_{2t-2})/(H_{2t-1}/H_{2t-2})|=|G/H_{2t-1}|=|S_0/H_{2t-1}|$.
Again, by (T3), $|G/H_{2t-2}| = |S_0/H_{2t-1}||H_{2t-1}/H_{2t-2}|=|S_0/H_{2t-2}||H_{2t-1}/H_{2t-2}|$. Combining (T1) and Lemma \ref{le:lesdtgd}, we obtain that $G/H_{2t-2} = (S_0/H_{2t-2}) \oplus (H_{2t-1}/H_{2t-2})$. Using Theorem \ref{le:legraphiso} and (T2), we see that there exists a subset $D\subset G$ such that $G/H_{2t-3} = (S_0/H_{2t-3}) \oplus (D/H_{2t-3})$. Moreover,
$$
(G/H_{2t-4})/(H_{2t-3}/H_{2t-4}) = \left((S_0/H_{2t-4})/(H_{2t-3}/H_{2t-4})\right) \oplus \left((D/H_{2t-4})/(H_{2t-3}/H_{2t-4})\right).
$$
Again, using Lemma \ref{le:lesdtgd} and (T1), we obtain that $G/H_{2t-4} = (S_0/H_{2t-4}) \oplus ((D/H_{2t-4}) + (H_{2t-3}/H_{2t-4}))$. Since $t$ is finite, repeating this process we can eventually obtain a subset $D' \subset G$ containing $0$ such that $G/H_1 = (S_0/H_1) \oplus (D'/H_1)$, and a further application of Theorem \ref{le:leobr07} using (T1) with $i=1$ gives $G = S_0\oplus (D'+H_1)$ as desired.

In the case when $S_0$ is period, the proof is essentially the same as above with $G/H_0, S_0/H_0$ playing the role of $G, S_0$, respectively, where $H_0$ is the subgroup of periods of $S_0$ in $G$.

Up to now we have proved the necessary and sufficient condition in Theorem \ref{th:thplp}.

We now prove (a) and (b). Suppose that $p^{\ell}$ divides $n$, $S_0$ is pyramidal and aperiodic, and the longest admissible subgroup series associated with $S_0$ is $H_0 < H_1 < H_2 < \dots < H_{2t-1} \le H_{2t} = \ZZZ_n$. Then $n = p^{\ell} m$, where $m \ge 2$ as $\Ga$ is non-complete, and $H_0=\{0\}$ as $S_0$ is aperiodic. By what we have proved above, $\Gamma$ admits perfect codes.

(a) Assume $(S_0-S_0)\cap\langle p^\ell \rangle=\{0\}$. Consider the case $t=1$ first. In this case we have $H_1 = \langle p^\ell \rangle$ and $\cay(\ZZZ_n/H_1, (S_0/H_1) \setminus \{H_1\})$ is a complete graph with $0+H_1$ as the unique perfect code containing the identity element of $\ZZZ_n/H_1$. By Lemma \ref{le:ledeng1424}, we have $\ZZZ_n=S_0\oplus H_1$. Hence $H_1$ is the only perfect code of $\Ga$ containing $0$.

Now consider the case $t>1$. In this case $S_0/H_1$ is periodic and its subgroup of periods has order $p^{\ell_1}$ for some integer $0<\ell_1<\ell$. Thus, $\cay(\ZZZ_n/H_1, (S_0/H_1) \setminus \{H_1\})$ admits a perfect code $D/H_1$ which is not a subgroup of $\ZZZ_n/H_1$. So, by Theorem \ref{le:leobr07} and the definition of a pyramidal set, $\Ga$ admits $D+H_1$ as a perfect code which is not a subgroup of $\ZZZ_n$. On the other hand, since $(S_0-S_0)\cap\langle p^\ell\rangle=\langle 0\rangle$ and $|S_0|\cdot|\langle p^\ell\rangle|=p^\ell m$, $S_0+\langle p^\ell\rangle$ is a direct sum and $S_0\oplus\langle p^\ell\rangle=\ZZZ_n$ by Lemma \ref{le:ledirectsum}(c), and so $\langle p^\ell \rangle$ is a perfect code in $\Gamma$. Thus the subgroup $\langle p^\ell \rangle$ of $\ZZZ_n$ with order $n/p^{\ell}$ is a perfect code in $\Gamma$.

(b) Assume $(S_0-S_0)\cap\langle p^\ell \rangle\neq\{0\}$. Then by Theorem \ref{le:leobr07}, $\langle p^\ell \rangle$ is not a perfect code in $\Gamma$. Since any perfect code in $\Ga$ has size $n/p^\ell$ and $\langle p^\ell \rangle$ is the only subgroup of $\ZZZ_n$ with order $n/p^\ell$, it follows that $\Gamma$ admits no subgroup of $\ZZZ_n$ as a perfect code. However, $\Gamma$ admits perfect codes as mentioned earlier. Hence there are perfect codes in $\Ga$ but none of them is $\langle p^\ell \rangle$ .
\qed
\end{Proof}

We illustrate the two cases in Theorem \ref{th:thplp} by two examples.

\begin{example}
\label{ex:ex3.4}
Consider the circulant graph $\cay(\ZZZ_{3^2\cdot 10}, S)$ with order $3^2\cdot 10$ and degree $3^2 - 1$, where $S = \{1,5,6,7,83,84,85,89\}$. It can be verified that $S_0$ is aperiodic and $(S_0 - S_0) \cap \langle 3^2 \rangle = \{0\}$. Setting $H_0=\{0\}$, $H_1=\langle 18\rangle$, $H_2=\langle 6\rangle$ and $H_3=\langle 3\rangle$, we see that $S_0$ is a pyramidal set of $\ZZZ_{3^2\cdot 10}$ with $H_0<H_1<H_2<H_3 < H_4 = \ZZZ_{3^2\cdot 10}$ as an admissible subgroup series. One can verify that $\langle 3^2 \rangle$ is a perfect code in $\cay(\ZZZ_{3^2\cdot 10}, S)$ as guaranteed by part (a) of Theorem \ref{th:thplp}. Since $t = 2 > 1$ in this case, by part (a) of Theorem \ref{th:thplp} there should have perfect codes in $\cay(\ZZZ_{3^2\cdot 10}, S)$ containing $0$ which are different from $\langle 3^2 \rangle$. Indeed, this is the case as $D = \{0,3,18,21,36,39,54,57, 72,75\} \neq \langle 3^2 \rangle$ is such a perfect code in $\cay(\ZZZ_{3^2\cdot 10}, S)$. Note that $D$ is not equally spaced. One can verify that $\cay(\ZZZ_{3^2\cdot 10}, S)$ is not isomorphic to the wreath product of a smaller circulant graph admitting perfect codes and a complete graph with at least two vertices. This shows that there exist circulant graphs which admit a non-equally spaced perfect code but are not the wreath product of a smaller circulant graph admitting perfect codes and a complete graph with at least two vertices.

Setting $L_0 = \{0\}$ and $L_1 = \langle 9\rangle$, we see that $L_0 < L_1 < L_2 = \ZZZ_{3^2\cdot 10}$ is another admissible subgroup series associated with $S_0$. This shows that the same extended connection set can be pyramidal with respect to different admissible subgroup series.
\qed
\end{example}

\begin{example}
Let
\begin{align*}
S = \{1, 44, 45, 46, 224, 225, 226, 269, 270, 271, 314, 315, 316, 2114, 2115,\\  2116, 2159, 2160, 2161, 2204, 2205, 2206, 2384, 2385, 2386, 2429\}.
\end{align*}
Then $\cay(\ZZZ_{3^3\cdot 90}, S)$ is a circulant graph with order $3^3\cdot 90$ and degree $3^3 - 1$. We have $(S_0 - S_0) \cap \langle 3^3 \rangle \ne \{0\}$ and $S_0$ is aperiodic. It can be verified that $S_0$ is a pyramidal set of $\ZZZ_{3^3\cdot 90}$ with $H_0 < H_1 < H_2 < H_ 3 < H_4 < H_5 < H_6=\ZZZ_{3^3\cdot 90}$ as an admissible subgroup series, where $H_0=\{0\}$, $H_1=\langle 810\rangle$, $H_2=\langle 270\rangle$, $H_3=\langle135\rangle$, $H_4=\langle 45\rangle$ and $H_5=\langle 3\rangle$. Thus, by part (b) of Theorem \ref{th:thplp}, there are perfect codes in $\cay(\ZZZ_{3^3 \cdot 90}, S)$ containing $0$, but none of them is $\langle 3^3 \rangle$. It can be verified that
\begin{align*}
D=\{0, 3,6,9,12,15,18,21,24,27,30,33,36,39,42, 135, 138, 141, 144,\\ 147, 150, 153, 156, 159, 162, 165, 168, 171, 174, 177\}+\langle 810\rangle
\end{align*}
is such a perfect code.

Note that $\{0, 270\} \subset S_0$ and $270-0\equiv 0\pmod{3^3}$. So this example shows that a circulant graph $\cay(\ZZZ_n, S)$ with $|S_0|=p^\ell$ can admit perfect codes even if there are distinct elements $s, s'\in S_0$ such that $s-s'\equiv 0\pmod{p^\ell}$.
\qed
\end{example}


\section{Proof of Theorem \ref{th:th2.2.1}}
\label{sec:const}

A graph $\Si$ is a \emph{covering graph} of a graph $\Ga$ if there exists a surjective map $\rho: V(\Si) \rightarrow V(\Ga)$ such that for each $v \in V(\Si)$ the restriction of $\rho$ to the neighourhood of $v$ in $\Si$ is a bijection onto the neighbourhood of $\rho(v)$ in $\Ga$. Call $\rho$ a \emph{covering map} from $\Si$ to $\Ga$. When all fibres $\rho^{-1}(u)$, $u \in V(\Ga)$ have the same size, say, $k$, we also say that $\Si$ is a \emph{$k$-fold cover} of $\Ga$.

\begin{defn}
\label{def:fgm}
Let $\Si=\cay(\ZZZ_n, S)$ be a circulant graph. Let $m$ be a divisor of $n$ and $H = \langle n/m \rangle$ the unique subgroup of $\ZZZ_n$ with order $m$. If $(S_0-S_0)\cap H = \{0\}$, then define
$$
f_m(\Si)=\cay(\ZZZ_{n}/H, S/H);
$$
if $S_0$ is $H$-periodic, then define
$$
g_{m}(\Si)=\cay(\ZZZ_{n}/H, (S_0/H) \setminus \{H\}).
$$
\end{defn}

\begin{remark}
\label{rem:fgm}
(i) If $(S_0-S_0)\cap H = \{0\}$, then $f_m(\Si) \cong \cay(\ZZZ_{n/m},T)$, where $T = \{s \pmod {(n/m)}: s \in S_0\} \setminus \{0\} = \{s \pmod {(n/m)}: s \in S\}$. (Note that $|T| = |S/H| = |S|$.) Since $(S_0-S_0)\cap H = \{0\}$, each coset of $H$ in $\ZZZ_n$ is an independent set of $\Si$ and so $\Si$ is an $m$-fold cover of $f_m(\Si)$. We may view $f_m$ as the canonical projection from $\ZZZ_n$ to $\ZZZ_{n}/H$ which sends $x \in \ZZZ_n$ to $x+H \in \ZZZ_{n}/H$, so that $f_m$ is a covering map from $\Si$ to $f_m(\Si)$. If $\Si$ is connected, then so is $f_{m}(\Si)$.

(ii) If $S_0$ is $H$-periodic, then $g_{m}(\Si) \cong \cay(\ZZZ_{n/m},T)$, where $T = \{s \pmod {(n/m)}: s \in S_0\} \setminus \{0\}$. (Note that $|T| = (|S_0|/|H|)-1$.) Again, we may view $g_{m}$ as the canonical projection from $\ZZZ_n$ to $\ZZZ_{n}/H$ which sends $x \in \ZZZ_n$ to $x+H \in \ZZZ_{n}/H$, so that $g_{m}$ is a graph homomorphism from $\Si$ to $g_{m}(\Si)$.
\end{remark}

\begin{construction}
\label{const:2}
Let $n=p^\ell m$, where $p$ is a prime and $\ell\geq 1, m\geq 2$ are integers. Let $\Ga = \cay(\ZZZ_n,S)$ be a connected circulant graph with order $n$ and degree $p^\ell - 1$ such that $S_0$ is aperiodic and pyramidal. Let $\iota: H_0<H_1<\cdots<H_{2t-1} \le H_{2t} = \ZZZ_n$ be an admissible subgroup series associated with $S_0$, where $t \ge 1$ and $H_0 = \{0\}$. By Lemma \ref{lem:iota}, this series determines two sequences of positive integers (see \eqref{eq:lm}), namely $\mathbf{l}_{\iota} = (\ell_1, \ell_2, \ldots, \ell_{t})$ and $\mathbf{m}_{\iota} = (m_1, m_2, \ldots, m_{t})$, such that $\sum_{i=1}^t \ell_i = \ell$, $\prod_{i=1}^t m_i$ divides $m$, $|H_{2i}/H_{2i-1}|=p^{\ell_i}$ and $|H_{2i-1}/H_{2i-2}|=m_i$ for $1 \leq i \leq t$, and $\cay(\ZZZ_n/H_{2t-1}, (S_0/H_{2t-1})\setminus \{H_{2t-1}\}) \cong K_{p^{\ell_t}}$. Set
$$
\Gamma_i = \cay(\ZZZ_n/H_i, (S_0/H_i)\setminus \{H_i\})
$$
for $0 \leq i \leq 2t-1$. Then each $\Gamma_i$ is a connected circulant graph. We have $\Gamma_0 = \Ga$ (as $H_0 = \{0\}$) and $\Gamma_{2t-1} \cong K_{p^{\ell_t}}$.

Let
$$
f_{m_i}: \ZZZ_n/H_{2(i-1)}\to \ZZZ_n/H_{2i-1}
$$
for $1 \le i \le t$ and
$$
g_{p^{\ell_i}}: \ZZZ_n/H_{2i-1}\to \ZZZ_n/H_{2i}
$$
for $1 \le i \le t-1$ be the natural projections. Slightly abusing notation, for any graph $\Si$ with vertex set $\ZZZ_n/H_{2(i-1)}$ we use $f_{m_i}(\Si)$ to denote the graph with vertex set $\ZZZ_n/H_{2i-1}$ and adjacency relation inherited from $\Si$ (thus, $f_{m_i}$ becomes a graph homomorphism from $\Si$ to $f_{m_i}(\Si)$). For any graph $\Si$ with vertex set $\ZZZ_n/H_{2i-1}$, the graph $g_{p^{\ell_i}}(\Si)$ is similarly defined. In particular, $f_{m_i}(\Ga_{2(i-1)})$ is well defined for $1 \le i \le t$ and $g_{p^{\ell_i}}(\Ga_{2i-1})$ is well defined for $1 \le i \le t-1$. So we have the chain of graph homomorphisms:
$$
\Gamma_0 \xrightarrow{f_{m_1}} \Gamma_1 \xrightarrow{g_{p^{\ell_1}}} \Gamma_2 \xrightarrow{} \cdots \xrightarrow{} \Gamma_{2(t-2)} \xrightarrow{f_{m_{t-1}}} \Gamma_{2t-3} \xrightarrow{g_{p^{\ell_{t-1}}}} \Gamma_{2(t-1)} \xrightarrow{f_{m_t}} \Gamma_{2t-1}.
$$
This means that $\Ga$ can be reduced to the complete graph $\Gamma_{2t-1} \cong K_{p^{\ell_t}}$ through the $2t-1$ projections in this chain. We may present this fact as
$$
(f_{m_t} g_{p^{\ell_{t-1}}} f_{m_{t-1}} \cdots g_{p^{\ell_1}} f_{m_1})(\Ga) \cong K_{p^{\ell_t}}.
$$
Note that $f_{m_i}$ is a covering map from $\Ga_{2(i-1)}$ to $\Gamma_{2i-1}$ as $\iota$ and $S_0$ satisfy (T1)--(T3).
\qed
\end{construction}

An important remark is in order: Since $\iota$ and $S_0$ satisfy (T1)--(T3), one can verify that $f_{m_i}(\Si)$ and $g_{p^{\ell_i}}(\Si)$ defined in Construction \ref{const:2} are the same as $f_{m_i}(\Si)$ and $g_{p^{\ell_i}}(\Si)$ obtained from Construction \ref{const:3} below for appropriate $\Si$, respectively. So there is no ambiguity to use the graph homomorphisms $f_m$ and $g_m$ defined in Definition \ref{def:fgm}.

We now introduce two other maps which are right-inverses of $f_m$ and $g_m$, respectively.

\begin{defn}
\label{def:fgmbar}
Let $\Gamma = \cay(\ZZZ_n, S)$ and let $m \ge 1$ be an integer. A map $\sigma: S_0 \rightarrow \{0, 1, \ldots, m-1\}$ is said to be \emph{feasible} if $\sigma(0)=0$ and $-T(\sigma) = T(\sigma)$ in $\ZZZ_{mn}$, where $T(\sigma) = \{s+\sigma(s)n: s \in S\}$. Set $T(\sigma)_0 = T(\sigma) \cup \{0\} = \{s+\sigma(s)n:s\in S_0\}$.

Given a feasible map $\sigma: S_0 \rightarrow \{0, 1, \ldots, m-1\}$, define
$$
\Ga(\sigma) = \cay (\ZZZ_{mn}, T(\sigma)).
$$
Define
$$
\baf_{m}(\Gamma) = \{\Ga(\sigma): \s: S_0 \rightarrow \{0, 1, \ldots, m-1\} \text{ is feasible}\}
$$
$$
\baf_m(S_0) = \{T(\sigma)_0: \s: S_0 \rightarrow \{0, 1, \ldots, m-1\} \text{ is feasible}\}.
$$
Define
$$
\bag_{m}(\Gamma) = \cay (\ZZZ_{mn}, T)
$$
to be the Cayley graph of $\ZZZ_{mn}$ with extended connection set $T_0 = \{s+in: s\in S_0, 0 \le i \le m - 1\}$. Write $T_0 = \bag_{m}(S_0)$.

We often write $\baf_{m}(S)$, $\bag_{m}(S_0)$ in place of $\baf_{m}(\Gamma), \bag_{m}(\Gamma)$, respectively.
\end{defn}

\begin{remark}
\label{rem:fgmbar}
(i) By definition $\baf_m(S_0)$ is the family of extended connection sets of the circulant graphs in $\baf_{m}(\Gamma)$. More specifically, $\cay (\ZZZ_{mn}, T) \in \baf_m(\Ga)$ if and only if $T_0 \in \baf_m(S_0)$.

Note that $(T(\sigma)_0 - T(\sigma)_0) \cap \langle n \rangle = \{0\}$ and $\pi: \ZZZ_{mn} \rightarrow \ZZZ_{mn}/\langle n \rangle, x \mapsto x+\langle n \rangle$ is a covering map from $\Si(\sigma)$ to $\Ga$, where $\langle n \rangle \cong \ZZZ_m$ is the subgroup of $\ZZZ_{mn}$ generated by $n$. Therefore, $\bar{f}_m(\Gamma)$ can be described as the set of circulant graphs $\Sigma=  \cay (\ZZZ_{mn}, T)$ such that the natural projection $\pi:\ZZZ_{mn}\to \ZZZ_n$ is an $m$-covering of $\Sigma$ onto $\Gamma$. In this way $\bar{f}_m$ can be seen as a right  inverse of $f_m$ as $f_m(\bar{f}_m(\Gamma))=\Gamma$.  

(ii) We observe that  a graph $\Sigma$ in $\bar{f}_m(\Gamma)$ needs not be connected (and it is certainly disconnected if $\Gamma$ is not connected). However, if $\Gamma$ is connected, then $\bar{f}_m(\Gamma)$ contains connected graphs. Write  $S$ as the disjoint union $S=S_1\cup (-S_1)$ and   define $\sigma (s)=1$ for $s\in S_1$ and $\sigma (-s)=m-1$, so that $\sigma$ is feasible. The connectivity of  $\Gamma$ implies $\gcd (S_1\cup \{n\})=1$ and then $\gcd ((S_1+m)\cup \{nm\})= 1$ as well, implying that $\Gamma (\sigma)$ is connected.

(iii) $\Ga$ is the quotient graph of $\bag_{m}(\Gamma)$ with respect to the partition $\{i+\langle n \rangle: 0 \le i \le n-1\}$ of $\ZZZ_{mn}$. Moreover, $\bag_{m}(\Gamma) = \Ga \otimes \cay(\ZZZ_{m}, \ZZZ_{m} \setminus \{0\}) \cong \Ga \otimes K_{m}$ and the subgroup of periods of $T_0$ in $\ZZZ_{mn}$ is $\langle n\rangle$. Hence, if $\Gamma$ is connected, then so is $\bag_m(\Gamma)$.

(iv) We note that, if $n$ is even and $|S|$ is odd, then for every feasible map $\sigma$, $|T(\sigma)|$ is odd and therefore $T(\sigma)$ contains the element $mn/2=s+\sigma (s)m$. If $m$ is also even, as $0 \le \sigma (s)\le m-1$ we must have $\sigma (s)=m/2$ and $s=0\in S$, which is not possible. Hence, we must have $m$ odd in this case.
\end{remark}

\begin{construction}
\label{const:3}
Let $n=p^\ell m$, where $p$ is a prime and $\ell\geq 1$, $m\geq 2$ are integers. Let $\mathbf{l} = (\ell_1, \ell_2, \ldots, \ell_{t})$ and $\mathbf{m} = (m_1, m_2, \ldots, m_t)$ be sequences of positive integers such that $\sum_{i=1}^t \ell_i =\ell$ and $\prod_{i=1}^{t} m_i = m$. Set $r=\ell_t$. Consider the following $2t-1$ ``lifts'' of $K_{p^r}$: First, choose a connected graph $\Gamma_t$ from $\baf_{m_{t}}(K_{p^r})$; set $\Sigma_{t-1}=\bag_{p^{\ell_{t-1}}}(\Gamma_t)$ and choose a connected graph $\Gamma_{t-1}$ from $\baf_{m_{t-1}}(\Sigma_{t-1})$; iterate this procedure till, finally, set $\Sigma_{1} = \bag_{p^{\ell_1}}(\Gamma_{2})$ and choose a connected graph $\Gamma_{1}$ from $\baf_{m_1}(\Sigma_{1})$. Note that, as noted in Remark \ref{rem:fgmbar} (ii), $\baf_{m_i}(\Sigma_{i})$ contains connected graphs, which are the ones chosen in this procedure.

Denote by $\GG_{\mathbf{l}, \mathbf{m}}(n, p^\ell)$ the family of circulant graphs $\Gamma_1$ (as above) obtained from this construction. In view of Definition \ref{def:fgmbar} and Remark \ref{rem:fgmbar}, we see that all graphs in $\GG_{\mathbf{l}, \mathbf{m}}(n, p^\ell)$ are connected with order $n$ and degree $p^\ell - 1$. Symbolically,
$\GG_{\mathbf{l}, \mathbf{m}}(n, p^\ell)$ is obtained from applying $\baf_{m_1}\bag_{p^{\ell_1}} \dots \baf_{m_{t-1}}\bag_{p^{\ell_{t-1}}}  \baf_{m_t}$ to $K_{p^{\ell_t}}$.
\qed
\end{construction}

The following lemma can be easily proved.

\begin{lemma}
\label{pr:pr4.1}
Let $p$ be a prime, $m, \ell \ge 1, n \ge 3$ integers, and $\Gamma=\cay(\ZZZ_n, S)$ a circulant graph of order $n$. Then the following hold:
\begin{itemize}
\item[(a)] if $|S_0|$ is a divisor of $n$ and $(S_0-S_0)\cap\langle|S_0|\rangle=\{0\}$ in $\ZZZ_n$, then $(T(\s)_0-T(\s)_0)\cap\langle|T(\s)_0|\rangle=\{0\}$ in $\ZZZ_{mn}$ for any $T(\s)_0 \in \baf_m(S_0)$;
\item[(b)] if $|S_0|$ is a divisor of $n$, then $\{x \pmod{|S_0|}: x \in S_0\} = \{x \pmod{|T(\s)_0|}: x \in T(\s)_0\}$ for any $T(\s)_0 \in \baf_m(S_0)$;
\item[(c)] if $|S_0|$ is a divisor of $n$, $p$ is a divisor of $n/|S_0|$, and $(S_0-S_0)\cap\langle|S_0|\rangle=\{0\}$ in $\ZZZ_n$, then $(T_0-T_0) \cap \langle |T_0| \rangle \neq \{0\}$ in $\ZZZ_{p^{\ell}n}$ for $T_0 = \bag_{p^{\ell}}(S_0)$.
\end{itemize}
\end{lemma}

Set
$$
f_m(\baf_m(\Gamma))=\{f_m(\Ga(\sigma)): \Ga(\sigma)\in \baf_m(\Gamma)\},\;
\baf_m(f_m(\Gamma)) = \baf_m(\Si) \text{ where } \Sigma = f_m(\Gamma).
$$

\begin{lemma}
\label{pr:pr4.2}
Let $n \ge 3$ be an integer, $m \ge 1$ a divisor of $n$, and $\Gamma=\cay(\ZZZ_n, S)$ a circulant graph of order $n$. Then the following hold:
\begin{itemize}
\item[(a)] if $(S_0-S_0) \cap \langle n/m \rangle = \{0\}$ in $\ZZZ_n$, then all graphs in $f_m(\baf_m(\Gamma))$ are isomorphic to $\Gamma$, and moreover $f_m(\baf_m(\Gamma)) \subseteq \baf_m(f_m(\Gamma))$;
\item[(b)] if $S_0$ is $\langle n/m \rangle$-periodic, then $\Gamma \cong g_m(\bag_{m}(\Gamma)) \cong \bag_m(g_m(\Gamma))$.
\end{itemize}
\end{lemma}

\begin{proof}
(a) If $(S_0-S_0)\cap\langle n/m \rangle=\{0\}$ in $\ZZZ_n$, then by the definitions of $f_m$ and $\baf_m$ we see that all graphs in  $f_m(\baf_m(\Gamma))$ are isomorphic to $\Gamma$ and $f_m(\baf_m(\Gamma))\subseteq \baf_m(f_m(\Gamma))$.

(b) Denote $k = n/m$. Assume that $S_0$ is $\langle k\rangle$-periodic. Then $\Gamma \cong g_m(\bag_m(\Gamma))$ by the definitions of $g_m$ and $\bag_m$. Recall that $g_{m}(\Gamma) = \cay(\ZZZ_{n}/\langle k \rangle, (S_0/\langle k \rangle) \setminus \{\langle k \rangle\}) \cong \cay(\ZZZ_{k},T)$, where $T$ satisfies $T_0 = \{s \pmod k: s \in S_0\}$. So the connection set of $\bag_m(g_m(\Gamma))$ together with the identity element is given by $\{s \pmod k + ik: s\in S_0, 0 \le i \le m - 1\}$, which is identical to $S_0$ since $S_0$ is $\langle k\rangle$-periodic (and hence $S_0$ is the union of some $\langle k \rangle$-cosets in $\ZZZ_n$). Hence $\bag_m(g_m(\Gamma)) \cong \Ga$.
\end{proof}

\begin{lemma}
\label{it:it2.1.7}
Let $n=p^\ell m$, where $p$ is a prime and $\ell\geq 1, m\geq 2$ are integers. Let $\cay(\ZZZ_n, R)$ and $\cay(\ZZZ_n,S)$ be distinct circulant graphs of $\ZZZ_n$ with degree $|R|=|S|=p^\ell-1$. Then $\baf_{m}(R_0)\cap\baf_{m}(S_0)=\emptyset$.
\end{lemma}

\begin{proof}
Since $\cay(\ZZZ_n, R)$ and $\cay(\ZZZ_n, S)$ are distinct circulant graphs of $\ZZZ_n$, we have $R \neq S$. Thus one of the following cases happens.

\smallskip
\textsf{Case 1.} $\{x \pmod{|R_0|}: x \in R_0\} \neq \{x \pmod{|S_0|}: x \in S_0\}$.
\smallskip

In this case, by part (b) of Lemma \ref{pr:pr4.1}, for any $T(\d)_0\in \baf_m(R_0)$ and $T(\s)_0\in\baf_{m}(S_0)$, we have $\{x \pmod{|T(\d)_0|}: x \in T(\d)_0\} \neq \{x \pmod{|T(\s)_0|}: x \in T(\s)_0\}$ and hence $T(\d)_0\neq T(\s)_0$. It follows that $\baf_m(R_0)\cap \baf_m(S_0)=\emptyset$.

\smallskip
\textsf{Case 2.} $\{x \pmod{|R_0|}: x \in R_0\} = \{x \pmod{|S_0|}: x \in S_0\}$ but there are elements $r \in R_0, s \in S_0$ such that $r \neq s$ but $r \equiv s \pmod{p^\ell}$.
\smallskip

We claim that, for any $T(\d)_0\in \baf_{m}(R_{0})$ and $T(\s)_0\in\baf_m(S_0)$, the element $r + \sigma(r)n$ of $T(\d)_0$ and the element  $s + \t(s)n$ of $T(\s)_0$ are distinct. Suppose otherwise. Then $(s - r)+(\t(s)-\sigma(r))n=0$ and hence $s-r \equiv 0\pmod{n}$. Since $0 \leq r, s \leq n-1$, we must have $r = s$, which is a contradiction. So $r + \sigma(r)n \neq s + \t(s)n$ and thus $T(\d)_0 \neq T(\s)_0$. It follows that $\baf_m(R_0)\cap \baf_m(S_0)=\emptyset$.
\end{proof}

\begin{lemma}
\label{pr:pr4.3}
Let $n=p^\ell m$, where $p$ is a prime and $\ell\geq 1, m\geq 2$ are integers such that $\gcd(m,p)=1$, and let $\cay(\ZZZ_n, S)$ be a circulant graph with order $n$ and degree $|S|=p^\ell - 1$. Then the following hold for any integers $a, b \ge 2$ and $r \geq 1$:
\begin{itemize}
\item[(a)]
$(\baf_{b}\baf_{a})(S_0) = \baf_{ab}(S_0)$;
\item[(b)]
 $(\bag_{b}\bag_{a})(S_0) = \bag_{ab}(S_0)$;
\item[(c)]
if $\gcd(a,p)=1$, then $(\bag_{p^r}\baf_{a})(S_0)\subseteq (\baf_{a}\bag_{p^r})(S_0)$;
\item[(d)]
if $p$ is a divisor of $a$, then $(\bag_{p^r}\baf_{a})(S_0) \cap (\baf_{a}\bag_{p^r})(S_0) = \emptyset$;
\item[(e)]
if $p$ is a divisor of $a$, then for any nonnegative integers $c, h, k$ with $1\leq c<h$ and $k\neq 0$, we have $(\bag_{p^{k}}\baf_{a}\bag_{p^{h}})(S_0) \cap (\bag_{p^{k+c}}\baf_{a}\bag_{p^{h-c}})(S_0) = \emptyset$;
\item[(f)]
if $p$ is a divisor of $a$, then for any divisor $d$ of $a$ which is divisible by $p$, we have $(\baf_{bd}\bag_{p^r}\baf_{a/d})(S_0) \cap (\baf_{b}\bag_{p^r}\baf_{a})(S_0) = \emptyset$.
\end{itemize}
\end{lemma}

\begin{proof}
(a) Every member of $(\baf_{b}\baf_{a})(S_0)$ is of the form $T(\eta)_0\in\baf_{b}(T(\z)_0)$ for some $T(\z)_0\in \baf_{a}(S_0)$, where $\z: S_0 \rightarrow \{0,1,\ldots,a-1\}$ is feasible in $\ZZZ_{an}$ and $\eta: T(\z)_0 \rightarrow \{0,1,\ldots, b-1\}$ is feasible in $\ZZZ_{abn}$. We have $T(\eta)_0 = \{x+\eta(x)an: x \in T(\z)_0\} = \{s+\z(s)n+\eta(s+\z(s)n)an: s\in S_0\} = \{s+\s(s)n: s\in S_0\}$, where we set $\sigma(s) = \z(s)+\eta(s+\z(s)n)a$ for any $s\in S_0$. Since $0\leq \z(s)\leq a-1$ and $0\leq \eta(s+\z(s)n)\leq b-1$, we have $0\leq \sigma(s)\leq ab-1$ for any $s\in S_0$. Note that $\sigma(0)=\z(0)+\eta(0+\z(0)n)a=0+\eta(0+0)a=0$. Since $-T(\eta)_0 = T(\eta)_0$, we have $-T(\sigma)_0 = T(\sigma)_0$ and hence $\sigma: S_0 \rightarrow \{0,1,\ldots,ab-1\}$ is feasible in $\ZZZ_{abn}$. Moreover, by the definition of $T(\sigma)$ we have $T(\eta)_0=T(\sigma)_0 \in \baf_{ab}(S_0)$. Since this holds for any $T(\eta)_0\in (\baf_{b}\baf_{a})(S_0)$, it follows that $(\baf_{b}\baf_{a})(S_0) \subseteq \baf_{ab}(S_0)$.

Conversely, let $T(\sigma)_0\in\baf_{ab}(S_0)$, where $\sigma: S_0 \rightarrow \{0,1,\ldots,ab-1\}$ is feasible in $\ZZZ_{abn}$. Define $\d: S_0 \rightarrow\{0,1,\ldots,a-1\}$ in such a way that for any $s\in S_0$ we have $\d(-s) = a-1-\d(s)$ (and in particular $\d(0)=0$) in $\ZZZ_{an}$, and $\t(s) := (\sigma(s)-\d(s))/a$ is an integer between $0$ and $b-1$ and $\tau(-s)=b-1-\tau(s)$. (Since for each $s\in S_0$, $0\leq \sigma(s)\leq ab-1$, we have $\sigma(x)=x_s+y_sa$ for some integers $0\leq x_s \leq a-1$ and $0\leq y_s \leq b-1$, and thus $(\sigma(s)-x_s)/a$ is an integer between $0$ and $b-1$. Set $\d(s) = x_s$ and $\tau(s) = y_s$ to obtain the required $\d$.) Since $\delta(-s)=a-1-\delta(s)$ in $\ZZZ_{abn}$ for each $s\in S$, we have $T(\delta)=-T(\delta)$, and so by Definition \ref{def:fgmbar}, $\delta$ is feasible. It can be verified that $T(\d)_0\in\baf_{a}(S_0)$ and $\t$ is a feasible map from $S_0$ to $\{0,1,\ldots,b-1\}$. Note that $T(\sigma)_0 = \{s+\sigma(s)n:s\in S_0\} = \{s+(\d(s)+\t(s)a)n:s\in S_0\} = \{(s+\d(s)n)+\t(s)an:s\in S_0\}$ and $T(\d)_0=\{s+\d(s)n:s\in S_0\}$. Since $|S_0|=|T(\d)_0|$ and $\tau(s)=b-1-(\tau(-s))$ for any $s\in S_0$, we obtain a map $\nu:T(\d)_0\rightarrow\{0,1,\ldots,b-1\}$ such that $\nu(s+\d(s)n)=\t(s)$ for any $s \in S_0$ and $\nu(0)=\nu(0+\d(0)n)=\t(0)=0$. Since $T(\d)_0\in \baf_{a}(S_0)$ and $T(\sigma)$ is inverse-closed in $\ZZZ_{abn}$, $T(\nu)$ is also inverse-closed in $\ZZZ_{abn}$ and moreover $T(\nu)_0\in \baf_{b}(T(\d)_0)$. So $T(\sigma)_0 \in \baf_{b}(T(\d)_0) \subseteq (\baf_{b}\baf_{a})(S_0)$. Since this holds for any $T(\sigma)_0 \in \baf_{ab}(S_0)$, it follows that $\baf_{ab}(S_0) \subseteq (\baf_{b}\baf_{a})(S_0)$ and hence $(\baf_{a}\baf_{b})(S_0)=\baf_{ab}(S_0)$ as desired.

(b) Let  $T_0=\bag_{a}(S_0)$. Then $T_0=\{s+in:s\in S_0,0\leq i\leq a-1\}$ is a subset of $\ZZZ_{an}$ and $(\bag_{b}\bag_{a})(S_0) = \{x+jan: 0\leq j\leq b-1, x \in T_0\} = \{s+(i+ja)n:0\leq i\leq a-1, 0\leq j\leq b-1,s\in S_0\} = \{s+kn: 0\leq k\leq ab-1, s\in S_0\} = \bag_{ab}(S_0)$, where in the second last step we used the fact that $\{i+ja: 0\leq i\leq a-1, 0\leq j\leq b-1\}=\{k: 0\leq k\leq ab-1\}$. Thus $(\bag_{b}\bag_{a})(S_0)=\bag_{ab}(S_0)$ as desired.

(c) Assume $\gcd(a,p)=1$. We aim to prove that each element of $(\bag_{p^{r}}\baf_{a})(S_0)$ is an element of $(\baf_{a}\bag_{p^{r}})(S_0)$. Let $T_0\in (\bag_{p^{r}}\baf_{a})(S_0)$.
Then $T_0=\bag_{p^r}(T(\sigma)_0)$ for some feasible map $\sigma:S_0\rightarrow\{0,1,\ldots,a-1\}$ with $T(\sigma)_0\in\baf_{a}(S_0)$. So, by the definitions of $\baf_a$ and $\bag_{p^{r}}$, we have $T_0 = \{x+ian: x \in T(\sigma)_0,0\leq i\leq p^r-1\} = \{s+\sigma(s)n+ian:s\in S_0,0\leq i\leq p^r-1\}$.
For each fixed $s\in S_0$, set $A(s)=\{\a(i)n+\b(i)np^r: \a(i)\equiv (\sigma(s)+ia)\ (\text{mod}\,p^r), \b(i) = (\sigma(s)+ia-\a(i))/p^r, 0\leq i\leq p^r-1\}$. Then for any distinct $0\leq i, j \leq p^r-1$ we have $\a(i)\neq \a(j)$, for otherwise we would have $(i-j)an = (\sigma(s)n+ian)-(\sigma(s)n+jan) = \a(i)n+\b(i)np^r-(\a(j)+\b(j)np^r) = (\b(i)-\b(j))np^r$, which contradicts our assumption that $\gcd(a, p)=1$. Moreover, for each fixed $s\in S_0$, we have $\sigma(s)n+ian=\a(i)n+\b(i)np^r$ and $\{\a(i):0\leq i\leq p^{r}-1\}=\{j: 0\leq j\leq p^r-1\}$. So $T_0 = \{s+\a(i)n+\b(i)np^r:s\in S_0,0\leq i\leq p^r-1\} = \{s+\a(i)n+\b(i)np^r:s\in S_0,0\leq \a(i)\leq p^r-1\} = \{x+\b(i)np^r: x \in \bag_{p^r}(S_0)\}$. Note that $\b(i)=(\sigma(s)+a-\a(i))/p^r$ depends on $\a(i)$ and $s$.
Since $\gcd(a,p)=1$, for any $0\leq i\leq p^r-1$, $ia=\b(i)p^r$ if and only if $\b(i)=0$. Thus, if $s=0$ and $\a(i)=0$, then $ia=\b(i)p^r$ and hence $\b(i)=0$. So we have $\b(0)=0$. We can view $\b$ as a map from $\bag_{p^r}(S_0)$ to $\{0,1,\ldots,a-1\}$. Since $T_0$ is inverse-closed in $\ZZZ_{anp^r}$, $T_0 = \{x+\b(i)np^r: x \in \bag_{p^r}(S_0)\} = T(\b)$ is also inverse-closed in $\ZZZ_{anp^r}$ and hence $\b$ is a feasible map. By the definition of $\baf_{a}$, we obtain that $T_0\in \baf_{a}(\bag_{p^r}(S_0))$. Since this holds for any $T_0\in (\bag_{p^{r}}\baf_{a})(S_0)$, it follows that $(\bag_{p^r}\baf_{a})(S_0)\subseteq(\baf_{a}\bag_{p^r})(S_0)$.

(d) Suppose that $p$ divides $a$. Every element of $(\bag_{p^{r}}\baf_{a})(S_0)$ is of the form $X = \bag_{p^{r}}(T(\sigma)_0)$ for some $T(\sigma)_0 \in \baf_{a}(S_0)$, where $\sigma$ is a feasible map from $S_0$ to $\{0,1,\ldots,a-1\}$ in $\ZZZ_{an}$. So $X = \{x+ian: x \in T(\sigma)_0,0\leq i\leq p^r-1\} = \{s+\sigma(s)n+ian:s\in S_0,0\leq i\leq p^r-1\}$. Similarly, every element of $(\baf_{a}\bag_{p^{r}})(S_0)$ is of the form $Y = \{s+in+\t(s+in)np^r:s\in S_0,0\leq i\leq p^r-1\}$ for some map $\t: \bag_{p^r}(S_0) \rightarrow \{0,1,\ldots,a-1\}$ satisfying $\t(0)=0$ and $-Y = Y$ in $\ZZZ_{an p^r}$.

Set $A=\{x: x\in X, x\equiv0\pmod{n}\}$ and $B=\{y: y \in Y, y \equiv 0\pmod{n}\}$. Then $|A|=|B|=p^r$. Since each $x \in A$ is divisible by $n$ and is of the form $x = s+\sigma(s)n+ian$, we have $x = 0 + \sigma(0)n+ian = ian$ for some $0\leq i\leq p^r-1$. Similarly, each $y \in B$ is of the form $y = 0 + in + \t(0+in)np^r = in + \t(in)np^r$ for some $0\leq i\leq p^r-1$. It follows that there are elements $x \in A$ such that $x \neq 0$ and $x\equiv 0\pmod{anp^{r-1}}$. On the other hand, if an element $y \in B$ satisfies $y \equiv 0\pmod{anp^{r-1}}$, then $y = in + \t(in)np^{r} = kanp^{r-1}$ for some integers $0 \leq i \leq p^r-1$ and $k$, that is, $i+\t(in)p^{r}=kap^{r-1}$. Since $p$ is a divisor of $a$, we have $kap^{r-1}\equiv 0\pmod{p^r}$ and hence $i\equiv 0\pmod{p^r}$. Since $0 \leq i\leq p^r-1$, we must have $i=0$ and therefore $y = 0 + \tau(0)np^{r} = 0$. Thus, there is no element $y \in B$ such that $y \neq 0$ and $y \equiv 0\pmod{anp^{r-1}}$. Hence the number of elements $x$ of $X$ satisfying $x\equiv 0\pmod{anp^{r-1}}$ differs from the number of elements $y$ of  $Y$ satisfying  $y \equiv 0\pmod{anp^{r-1}}$. Therefore, $X \neq Y$. Since this holds for any $X \in (\bag_{p^r}\baf_{a})(S_0)$ and $Y \in (\baf_{a}\bag_{p^r})(S_0)$, we conclude that $(\bag_{p^r}\baf_{a})(S_0)\cap(\baf_{a}\bag_{p^r})(S_0)=\emptyset$.

(e) Every element of $(\bag_{p^{k}}\baf_{a}\bag_{p^{h}})(S_0)$ is of the form $X = \bag_{p^{k}}(T(\sigma)_0)$ for some $T(\sigma)_0 \in (\baf_{a}\bag_{p^{h}})(S_0)$, where $\sigma$ is a feasible map from $\bag_{p^{h}}(S_0)$ to $\{0,1,\ldots,a-1\}$ in $\ZZZ_{anp^{h}}$. So $X = \{s+in+\sigma(s+in)n p^{h}+janp^{h}: s\in S_0, 0\leq i\leq p^{h}-1,0\leq j\leq p^{k}-1\}$. Similarly, every element of $(\bag_{p^{k+c}}\baf_{a}\bag_{p^{h-c}})(S_0)$ is of the form $Y(\t) = \{s+in+\t (s+in)n p^{h-c}+janp^{h-c}:s\in S_0,0\leq i\leq p^{h-c}-1,0\leq j\leq p^{k+c}-1\}$, where $\t$ is a feasible map from $\bag_{p^{h-c}}(S_0)$ to $\{0,1,\ldots,a-1\}$ with $T(\t)_0 \in (\baf_{a}\bag_{p^{h-c}})(S_0)$. Set $A = \{x: x\in X, x\equiv0\pmod{n}\}$ and $B = \{y: y \in Y, y \equiv0\pmod{n}\}$. Then $|A|=|B|=p^{h+k}$. We have $A = \{in+\sigma(in)n p^{h}+janp^{h}: 0\leq i\leq p^{h}-1, 0\leq j\leq p^{k}-1\}$ and $B = \{in+\t(in)n p^{h-c}+janp^{h-c}: 0\leq i\leq p^{h-c}-1, 0\leq j\leq p^{k+c}-1\}$. Note that $in+\sigma(in)n p^{h}+janp^{h} \in A$ is divisible by $np^{h-c}$ if and only if $i$ is divisible by $p^{h-c}$. Similarly, $in+\t(in)n p^{h-c}+janp^{h-c} \in B$ is divisible by $np^{h-c}$ if and only if $i$ is divisible by $p^{h-c}$. Assume $c > 1$ first. Then there are exactly $\frac{p(p^c-1)}{p-1} p^{k}$ elements $x \in A$ satisfying $x\equiv 0\pmod{np^{h-c}}$, and there are exactly $p^{k+c}$ elements $y \in B$ satisfying $y \equiv 0\pmod{np^{h-c}}$. Since $c > 1$ and $p>1$, we have $\frac{p(p^c-1)}{p-1} p^{k}\neq p^{k+c}$ and therefore $X \neq Y$. Since this holds for any $X \in (\bag_{p^{k}}\baf_{a}\bag_{p^{h}})(S_0)$ and $Y \in (\bag_{p^{k+c}}\baf_{a}\bag_{p^{h-c}})(S_0)$, it follows that $(\bag_{p^{k}}\baf_{a}\bag_{p^{h}})(S_0)\cap(\bag_{p^{k+c}}\baf_{a}\bag_{p^{h-c}})(S_0)=\emptyset$.

Now assume $c=1$. Then there are exactly $p^{k}$ elements $x \in A$ satisfying $x\equiv 0\pmod{np^{h}}$. Since $p$ is a divisor of $a$, for any $y \in B$, $y \equiv 0\pmod{np^{h}}$ if and only if $in+\t(in)np^{h-1}\equiv 0\pmod{np^{h}}$ for some $0 \leq i \leq p^{h-1}-1$. Let $w$ be the number of such integers $i$ satisfying $in+\t(in)np^{h-1}\equiv 0\pmod{np^{h}}$.
Then there are exactly $w p^{k+1}$ elements $y \in B$ satisfying $y \equiv 0\pmod{np^{h}}$. Since $w p^{k+1}\neq p^{k}$, it follows that $X \neq Y$. Therefore, $(\bag_{p^{k}}\baf_{a}\bag_{p^{h}})(S_0)\cap(\bag_{p^{k+c}}\baf_{a}\bag_{p^{h-c}})(S_0)=\emptyset$.

(f) Assume $p$ is a divisor of $d$ and $d$ is a divisor of $a$. Then  $\baf_{a}(S_0)=\baf_{d}\baf_{a/d}(S_0)$ by (a). Since $p$ is a divisor of $d$, we have $(\bag_{p^r}\baf_{d}\baf_{a/d})(S_0)\cap(\baf_d\bag_{p^r}\baf_{a/d})(S_0)=\emptyset$ by (d). This together with Lemma \ref{it:it2.1.7} yields $(\baf_{b}\bag_{p^r}\baf_{d}\baf_{a/d})(S_0)\cap(\baf_{b}\baf_d\bag_{p^r}\baf_{a/d})(S_0)=\emptyset$. By (a) again, we obtain $(\baf_{bd}\bag_{p^r}\baf_{a/d})(S_0)\cap(\baf_{b}\bag_{p^r}\baf_{a})(S_0)=\emptyset$.
\end{proof}

Now we are ready to prove Theorem \ref{th:th2.2.1}. Set $\sum_{\mathbf{a}} = \sum_{i=1}^k a_i$ and $\prod_{\mathbf{a}} = \prod_{i=1}^{k} a_i$ for any sequence of positive integers $\mathbf{a} = (a_1, a_2, \ldots, a_k)$.

\smallskip

\begin{Proof}~\ref{th:th2.2.1}~
Let $n=p^\ell m$, where $p$ is a prime and $\ell\geq 1, m\geq 2$ are integers. We prove the statement in the theorem by induction on $\ell$.

Consider the base case $\ell=1$ first. Let $\Ga = \cay(\ZZZ_{pm}, S)$ be a connected circulant graph with order $n=pm$ and degree $|S|=p-1$ which admits a perfect code. Then by \cite[Theorem 3.1]{deng17}, $H = \langle p\rangle$ is a perfect code of $\Gamma$. Since $(S_0-S_0)\cap H = \{0\}$, by Definition \ref{def:fgm}, $f_m(\Gamma)$ is well defined. Moreover, $f_m(\Gamma)=K_p$ and $\Gamma\in \baf_m(f_m(\Gamma))=\baf_m(K_p) = \GG_{\mathbf{l}, \mathbf{m}}(n, p)$, where $\mathbf{l} = (\ell_1)$ with $\ell_1 = 1$ and $\mathbf{m} = (m_1)$ with $m_1 = m$. Note that the only pair of sequences $\mathbf{l}, \mathbf{m}$ satisfying $\sum_{\mathbf{l}} = 1$ and $\prod_{\mathbf{m}} = m$ is $\mathbf{l} = (\ell_1)$ with $\ell_1 = 1$ and $\mathbf{m} = (m_1)$ with $m_1=m$. Note also that $K_p=\cay(\ZZZ_p, \ZZZ_p\setminus\{0\})$ admits a perfect code and $x - y \not\equiv 0\pmod{p}$ for any distinct elements $x, y \in \ZZZ_p\setminus\{0\}$. Thus, by Lemma \ref{pr:pr4.1}(a), for any $\cay(\ZZZ_{pm}, S) \in \baf_m(K_p) = \GG_{\mathbf{l}, \mathbf{m}}(n, p)$, we have $x - y \not\equiv 0\pmod{p}$ for distinct $x, y \in S_0$. Hence, by Theorem \ref{le:ledeng17p}, $\cay(\ZZZ_{pm}, S)$ admits a perfect code. This establishes the base case for induction.

Assume that for some integer $k \ge 2$ the result holds for any $\ell$ up to $k-1$.
We aim to prove that the result is also true for $\ell=k$. Let $\Ga = \cay(\ZZZ_n, S)$ be a connected circulant graph with order $n=p^k m$ and degree $|S| = p^k - 1$ which admits a perfect code. Let $D$ be a perfect code in $\Ga$. Then by Lemma \ref{thm:primepower} either $D$ or $S_0$ is periodic. Consider first the case when $S_0$ is aperiodic. In this case $D$ is periodic, say, with the subgroup of periods $H$ of order $d$ for some divisor $d > 1$ of $m$. Then $H=\langle n/d \rangle$ and $\Si := \cay(\ZZZ_n/H, (S_0/H) \setminus \{H\})$ is isomorphic to some circulant graph $\cay(\ZZZ_{n/d}, T)$ with order $n/d$ and degree $|T| = p^k-1$.
Since $D/H$ is aperiodic, $T_0$ is periodic by Lemma \ref{thm:primepower}. Let $K$ be the subgroup of periods of $T_0$. Then $|K| = p^{h}$ for some $h \ge 1$, and $\De := \cay(\ZZZ_{p^{k-h}m/d},(T/K)\setminus \{K\})$ admits a perfect code.
Since $\De$ has degree $p^{k-h} - 1 < p^k - 1$, by the induction hypothesis we have $\De \in \GG_{\mathbf{l}_1, \mathbf{m}_1}(p^{k-h}m/d, p^{k-h})$ for some $\mathbf{l}_1 = (\ell_1, \ell_2, \ldots, \ell_t)$ and $\mathbf{m}_1 = (m_1, m_2, \ldots, m_t)$ such that $\sum_{\mathbf{l}_1} =k-h$ and $\prod_{\mathbf{m}_1} = m/d$. Note that $\De = g_{p^{h}}(f_{d}(\Gamma))$ and hence $\Gamma \in (\baf_{d}\bag_{p^{h}})(\De)$. Therefore, $\Ga \in \GG_{\mathbf{l}, \mathbf{m}}(n, p^{k})$ for $\mathbf{l} = (h, \mathbf{l}_1) = (h, \ell_1, \ell_2, \ldots, \ell_t)$ and $\mathbf{m} = (d, \mathbf{m}_1) = (d, m_1, m_2, \ldots, m_t)$. Now consider the case when $S_0$ is periodic. In this case the subgroup of periods of $S_0$ is of the form $H = \langle p^{h}m\rangle$ for some $h \ge 1$. Since $|S_0/H|=p^{k-h}<p^k$, we have $g_{p^{h}}(\Gamma) \in \GG_{\mathbf{l}_2, \mathbf{m}_2}(p^{k-h}m, p^{k-h})$ for some $\mathbf{l}_2 = (\ell_1, \ell_2, \ldots, \ell_t)$ and $\mathbf{m}_2 = (m_1, m_2, \ldots, m_t)$ with $\sum_{\mathbf{l}_2} = k-h$ and $\prod_{\mathbf{m}_2} = m$. Therefore, $\Ga \in \GG_{\mathbf{l}, \mathbf{m}}(n, p^{k})$ for $\mathbf{l} = (h, \mathbf{l}_2) = (h, \ell_1, \ell_2, \ldots, \ell_t)$ and $\mathbf{m} = (1, \mathbf{m}_2) = (1, m_1, m_2, \ldots, m_t)$.

Now we prove that for any $\mathbf{l} = (\ell_1, \ell_2, \ldots, \ell_t)$ and $\mathbf{m} = (m_1, m_2, \ldots, m_t)$ with $\sum_{\mathbf{l}} = k$ and $ \prod_{\mathbf{m}} = m$, any $\Ga = \cay(\ZZZ_n, S) \in \GG_{\mathbf{l}, \mathbf{m}}(n, p^{k})$ admits a perfect code. Since $\Gamma\in \GG_{\mathbf{l}, \mathbf{m}}(n, p^{k})$, there exist integers $\ell_1, m_1$ and a circulant graph $\Sigma$ with degree $p^{k-\ell_1}$ such that $\Gamma\in \baf_{m_1}\bag_{p^{\ell_1}}(\Sigma)$. By the induction hypothesis, we may assume that $\Sigma=\cay(\ZZZ_{p^{k-\ell_1}m/m_1},T)\in \GG_{\mathbf{l'}, \mathbf{m'}}(p^{k-\ell_1}m/m_1, p^{k-\ell_1})$ and that $\Sigma$ admits a perfect code, where $\mathbf{l'}=(\ell_2,\ell_3,\ldots,\ell_t)$ and $\mathbf{m'}=(m_2, m_3,\ldots,m_t)$. We claim that any perfect code $D \subset \ZZZ_{p^{k-\ell_1}m/m_1}$ in $\Sigma$ is also a perfect code in $\bag_{p^{\ell_1}}(\Sigma)$. In fact, $\bag_{p^{\ell_1}}(\Sigma) = \cay(\ZZZ_{n/m_1}, R)$, where $R = \{w+i\ZZZ_{p^{k-\ell_1}m/m_1}: w \in T_0, 0\leq i\leq p^{\ell_1}\}$. For any two distinct vertices $u, v \in D$, if $u$ is adjacent to $v$ in $\bag_{p^{\ell_1}}(\Sigma)$, then there exist $w \in T$ and $i \in \ZZZ_{p^{\ell_1}}$ such that $u - v = w + (ip^{k-\ell_1}m/m_1)$. So $u-v=w \in T$ or $v-u= w \in T$ as $u, v \in \ZZZ_{p^{k-\ell_1}m/m_1}$, but this means that $D$ is not an independent set of $\Sigma$, a contradiction. Thus, $D$ must be an independent set of $\bag_{p^{\ell_1}}(\Sigma)$. For any vertex $x$ of $\bag_{p^{\ell_1}}(\Sigma)$ outside $D$, there exists $i$ between $0$ and $p^{\ell_1}-1$ such that $x = u+(ip^{k-\ell_1}m/m_1)$ for some $u \in \ZZZ_{p^{k-\ell_1}m/m_1}$. So $u$ is adjacent to a unique vertex $v$ in $D$ and $x$ is adjacent to $v$ due to the form of $R$. Hence $D$ is a perfect code in $\bag_{p^{\ell_1}}(\Sigma)$ as claimed. Moreover, we have $\ZZZ_{n/m_1}=R_0 \oplus D$. Since $(S_0-S_0)\cap\langle n/m_1\rangle=\{0\}$, by Lemmas \ref{le:ledeng1424} and \ref{le:lesdtgd} we have $\ZZZ_{n}=S_0\oplus (D+\langle n/m_1\rangle)$ and hence $\Ga$ admits $D+\langle n/m_1\rangle$ as a perfect code.

Up to now we have proved that based on the induction hypothesis the result is true for $\ell=k$. This completes the proof of the theorem by mathematical induction.
\qed
\end{Proof}


\section{Proof of Theorem~\ref{cor:cor2.2}}
\label{sec:enum}

As a preparation for the proof of Theorem~\ref{cor:cor2.2}, we now construct a special admissible subgroup series of $\ZZZ_n$ associated with a pyramidal extended connection set $S_0$. We will soon see that this series is the unique longest admissible subgroup series of $\ZZZ_n$ associated with $S_0$.

\begin{construction}
\label{const:1}
Let $n=p^\ell m$, where $p$ is a prime and $\ell\geq 1, m\geq 2$ are integers. Let $S$ be a connection set of $\ZZZ_n$ with $|S| = p^\ell-1$ such that $S_0$ is pyramidal. We now construct a special admissible subgroup series
\begin{equation}
\label{eq:iotaS}
\iota(S_0): H_0 < H_1 < \cdots < H_{2t-1} \le H_{2t} = \ZZZ_n
\end{equation}
associated with $S_0$ together with two sequences of nonnegative integers
\begin{equation}
\label{eq:hk}
\mathbf{h}(S_0) := (h_0, h_1, \ldots, h_{t-1}, h_t),\;
\mathbf{k}(S_0) := (k_0, k_1, \ldots, k_{t-1}, k_t)
\end{equation}
such that the following hold:

(i) $H_{2i-1}=\langle p^{h_{i-1}}k_{i}\rangle$ and $H_{2i}=\langle p^{h_i}k_i\rangle$ for $1 \le i \le t$;

(ii) $k_0 = m$, $k_t = 1$, and $k_i$ is a divisor of $k_{i-1}$ for $1 \le i \le t$;

(iii) if $S_0$ is aperiodic, then $0 = h_t < h_{t-1} < \cdots < h_1 < h_0 = \ell$;

(iv) if $S_0$ is periodic, then $0 = h_t < h_{t-1} < \cdots < h_1 < h_0 < \ell$;

(v) the sequences $\mathbf{l}_{\iota(S_0)}$ and $\mathbf{m}_{\iota(S_0)}$ defined in \eqref{eq:lm} are related to $\mathbf{h}(S_0)$ and $\mathbf{k}(S_0)$ by
$$
\ell_i = h_i - h_{i-1},\; m_i = k_{i-1}/k_i,\; 1 \le i \le t.
$$

First, let $k_0 = m$. If $S_0$ is aperiodic, then let $h_0 = \ell$; if $S_0$ is periodic, then let $1 \le h_0 < \ell$ be such that $\langle p^{h_0} k_0 \rangle$ is the subgroup of periods of $S_0$, where $h_0$ and $k_0$ are positive integers such that $\langle p^{h_0}k_0\rangle$ is the subgroup of periods of $S_0$. Set $H_{0} := \langle p^{h_0} k_0 \rangle$. Since $S_0$ is a pyramidal set, by Theorem \ref{th:thplp}, $\cay(\ZZZ_n, S)$ admits a perfect code, say, $D$. Hence $\ZZZ_n=S_0\oplus D$. Since $S_0/H_0$ is aperiodic and $|S_0/H_0|=p^{h_0}$, $D/H_0$ is periodic in $\ZZZ_n/H_0$. So there is a subgroup $H_1/H_0$ of $\ZZZ_n/H_0$ such that $(S_0/H_0-S_0/H_0)\cap (H_1/H_0)=\{0\}$ and $S_0/H_1$ is periodic in $\ZZZ_n/H_1$. Since $(S_0/H_0-S_0/H_0)\cap (H_1/H_0)=\{0\}$, we have $|S_0/H_0|=|S_0/H_1|$. Note that $|D/H_0|=m$ and $H_1/H_0$ is the subgroup of periods of $D/H_0$. Hence $|H_1|$ is a divisor of $m$ and so there is an integer $k_1$ such that $H_1 = \langle p^{h_0} k_1 \rangle$. Let $H_2/H_{1}$ be the subgroup of periods of $S_0/H_{1}$ in $\ZZZ_n/H_1$, where $H_2$ is a subgroup of $\ZZZ_n$ which contains $H_1$ as a proper subgroup. This defines $H_2$ uniquely, and moreover $H_{2} = \langle p^{h_1} k_1 \rangle$ for some $h_1 < h_0$.

Inductively, assume that for some $i \ge 1$ we have constructed $H_0, H_1, H_2, \ldots, H_{2i-1}, H_{2i}$ together with $h_0, h_1, \ldots, h_i$ and $k_0, k_1, \ldots, k_i$ such that $H_{2j}=\langle p^{h_j}k_j\rangle$ and $H_{2j-1}=\langle p^{h_{j-1}}k_{j}\rangle$ for $1 \le j \le i$. If $H_{2i} = \ZZZ_n$, then set $t = i$ and stop. If $H_{2i} < \ZZZ_n$, then similarly to the proof above we can show that there exists a positive integer $k_{i+1}$ which is a proper divisor of $k_i$ such that $H_{2i+1} := \langle p^{h_i} k_{i+1} \rangle$ has the following properties: $(S_0/H_{2i} - S_0/H_{2i}) \cap (H_{2i+1}/H_{2i}) = \{0\}$, $S_0/H_{2i+1}$ is periodic, and $|S_0/H_{2i}| = |S_0/H_{2i+1}|$. Let $H_{2i+2}/H_{2i+1}$ be the subgroup of periods of $S_0/H_{2i+1}$ in $\ZZZ_n/H_{2i+1}$, where $H_{2i+2}$ is a subgroup of $\ZZZ_n$ which contains $H_{2i+1}$ as a proper subgroup. This defines $H_{2i+2}$ uniquely, and moreover $H_{2i+2} = \langle p^{h_{i+1}} k_{i+1} \rangle$ for some $h_{i+1} < h_i$.

Continuing the process, we can eventually construct an admissible subgroup series \eqref{eq:iotaS} associated with $S_0$ and two sequences \eqref{eq:hk} such that all (i)--(v) are respected. In fact, by the construction we see that (i)--(iv) are satisfied. Also, by (i) and \eqref{eq:H} we obtain that for $1\le i\le t$, $|\langle p^{h_i}k_i\rangle/\langle p^{h_{i-1}}k_i\rangle|=p^{h_{i-1}-h_i}=p^{\ell_i}$ and $|\langle p^{h_{i-1}}k_i\rangle/\langle p^{h_{i-1}}k_{i-1}\rangle|=k_{i-1}/k_i=m_i$,  and hence (v) is satisfied as well.
\qed
\end{construction}

\begin{lemma}
\label{lem:longest}
With the notation in Construction \ref{const:1}, the special admissible subgroup series $\iota(S_0)$ as given in \eqref{eq:iotaS} is the unique longest admissible subgroup series of $\ZZZ_n$ associated with $S_0$.
\end{lemma}

\begin{proof}
We use the notation in Construction \ref{const:1}. We first prove that $\iota(S_0)$ is an admissible subgroup series of $\ZZZ_n$ associated with $S_0$. Consider the case when $S_0$ is aperiodic first. In this case, by Construction \ref{const:1}, we have $H_0=\langle p^{h_0} k_0\rangle=\{0\}$, and $(S_0/H_{2i}-S_0/H_{2i})\cap (H_{2i+1}/H_{2i})=\{0\}$, $|S_0/H_{2i}|=|S_0/H_{2i+1}|$ and $S_0/H_{2i+1}$ is $(H_{2i+2}/H_{2i+1})$-periodic for $1\le i\le t$. Since $H_{2t}=\ZZZ_n$, $S_0/H_{2t-1}$ is $(\ZZZ_n/H_{2t-1})$-periodic and  $S_0\subseteq \ZZZ_n$, we have $|S_0/H_{2t-1}|=|\ZZZ_n/H_{2t-1}|$ and hence $\iota(S_0)$ is an admissible subgroup series of $\ZZZ_n$. In the case when $S_0$ is periodic, the proof is essentially the same as in the aperiodic case with $\ZZZ_n/H_0$ and $S_0/H_0$ playing the role of $\ZZZ_n$ and $S_0$, respectively.

Now we prove that $\iota(S_0)$ is the unique longest admissible subgroup series of $\ZZZ_n$ associated with $S_0$. To achieve this it suffices to prove that any admissible subgroup series
\[
\iota': L_0<L_1<\cdots<L_{2r-1}\le L_{2r}=\ZZZ_n
\]
of $\ZZZ_n$ associated with $S_0$ such that $r \ge t$ must be equal to $\iota(S_0)$. In fact, we may write $L_{2i-1}=\langle p^{x_{i-1}}y_i\rangle$ and $L_{2i}=\langle p^{x_i}y_i\rangle$ for $1\leq i\leq r$, so that $\iota'$ is determined by the sequences $\mathbf{x} := (x_0, x_1, \ldots, x_{r-1}, x_r)$ and $\mathbf{y} := (y_0, y_1, \ldots, y_{r-1}, y_r)$. Note that $H_0 = L_0$ regardless of whether $S_0$ is periodic or aperiodic. Set $\alpha_0 = h_0-h_1$, $\beta_0=k_0/k_1$, $\alpha_1=x_0-x_1$ and $\beta_1=y_0/y_1$. Since there are circulant graphs $\cay(\ZZZ_{p^{h_1}k_1},T)$ and $\cay(\ZZZ_{p^{x_1}y_1},T')$ such that  $g_{p^{\alpha_0}}f_{\beta_0}(\Gamma)\cong\cay(\ZZZ_{p^{h_1}k_1},T)$ and $g_{p^{\alpha_1}}f_{\beta_1}(\Gamma)\cong\cay(\ZZZ_{p^{x_1}y_1},T')$, $S_0$ is a member of $\baf_{\beta_0}(\bag_{p^{\alpha_0}}(T_0))$ and $\Gamma$ is a member of $\baf_{\beta_1}(\bag_{p^{\alpha_1}}(T'_0))$. By the definitions of $\baf_m$ and $\bag_m$, there are feasible $\sigma, \tau$ such that $S_0=\{t+ip^{h_1}k_1+\sigma(t+ip^{h_1}k_1)p^{h_0}k_1:1\le i\le p^{\alpha_0}-1, t\in T_0\}$ and $S_0=\{t'+ip^{x_1}y_1+\tau(t'+ip^{x_1}y_1)p^{x_0}y_1:1\le i\le p^{\alpha_1}-1, t'\in T_0\}$. Since $s=p^{h_1}k_1+\sigma(p^{h_1}k_1)p^{h_0}k_1 \in S_0$, there exist $t'\in T'$ and $0\le i\le p^{\alpha_1}-1$ such that $s = t'+ip^{x_1}y_1+\tau(t'+ip^{x_1}y_1)p^{x_0}y_1$. Hence $p^{h_1}k_1-t'-ip^{x_1}y_1=\tau(t'+ip^{x_1}y_1)p^{x_0}y_1-\sigma(p^{h_1}k_1)p^{h_0}k_1$. Since $\bag_{p^{\alpha_0}}(T_0)$ is a subset of $\ZZZ_{p^{h_0}k_1}$, we have $-p^{h_1}k_1=p^{h_0}k_1-p^{h_1}k_1$. Similarly, $-(t'+ip^{x_1}y_1)=p^{x_0}y_1-(t'+ip^{x_1}y_1)$. Hence $p^{h_0}k_1-p^{h_1}k_1-\sigma(p^{h_1}k_1)p^{h_0}k_1=p^{x_0}y_1-(t'+ip^{x_1}y_1)-\tau(t'+ip^{x_1}y_1)p^{x_0}y_1$. That is, $p^{h_0}k_1-p^{x_0}y_1=p^{h_1}k_1-(t'+ip^{x_1}y_1)-\tau(t'+ip^{x_1}y_1)p^{x_0}y_1+\sigma(p^{h_1}k_1)p^{h_0}k_1$.
Combining this with $p^{h_1}k_1-t'-ip^{x_1}y_1=\tau(t'+ip^{x_1}y_1)p^{x_0}y_1-\sigma(p^{h_1}k_1)p^{h_0}k_1$, we obtain that $p^{h_0}k_1=p^{x_0}y_1=0$. Since $h_0=x_0=\ell$, it follows that $y_1=k_1$. So we have proved  $y_1=k_1$,  which implies $H_1=L_1$. Since  $S_0/H_1$ is $(H_2/H_1)$-periodic, we then obtain $H_2=L_2$ and so $x_1=h_1$. Let $i$ be the largest integer between $1$ and $t$ such that $H_{2j-1} = L_{2j-1}$, $H_{2j} = L_{2j}$, $x_j = h_j$ and $y_j = k_j$ for $1 \le j \le i$. (Note that $i$ is well defined as $H_1=L_1$, $H_2=L_2$, $x_1 = h_1$ and $y_1 = k_1$.) We claim that $i = t$. Suppose otherwise. Then $i < t \le r $. Consider the subgroup series $H_{2i}/H_{2i}<H_{2i+1}/H_{2i}<\cdots<H_{2t}/H_{2i}=\ZZZ_n/H_{2i}$ and $L_{2i}/L_{2i}<L_{2i+1}/L_{2i}<\cdots<L_{2r}/L_{2i}=\ZZZ_n/L_{2i}$. Let $H'_{j-2i}=H_j/H_{2i}$ for $2i\le j\le 2t$ and $L'_{j-2i}=L_j/L_{2i}$ for $2i\le j\le 2r$. Then we have two subgroup series $H'_{0}<H'_{1}<\cdots<H'_{2t-2i}=\ZZZ_n/H_{2i}$ and $L'_{0}<L'_{1}<\cdots<L'_{2r-2i}=\ZZZ_n/L_{2i}$. Since $H_{2i}=L_{2i}$, we have $S_0/H_{2i}=S_0/L_{2i}$. Set $S'_0=S_0/H_{2i}=S_0/L_{2i}$. Similarly to the above proof leading to $y_1=k_1$ and $H_1=L_1$, we can show that $H'_1=L'_{1}$. That is, $H_{2i+1}/H_{2i}=L_{2i+1}/L_{2i}$. So $H_{2i+1}=L_{2i+1}$ and $y_{i+1} = k_{i+1}$. Since $S_0/H'_{1}=S_{0}/L'_{1}$ is $(H'_{2}/H'_1)$-periodic and $S_0/H'_{1}=S_{0}/L'_{1}$ is $(L'_{2}/L'_1)$-periodic, from $H'_1=L'_1$ we obtain $H'_2=L'_2$ and $H_{2i+2}=L_{2i+2}$ as $H_{2i}=L_{2i}$. Hence $x_{i+1} = h_{i+1}$, but this contradicts the choice of $i$. So we must have $i=t$, which forces $t = r$ and $\iota' = \iota(S_0)$ as desired.
\end{proof}

\begin{construction}
\label{const:DD}
Let $n=p^\ell m$, where $p$ is a prime and $\ell\geq 1, m\geq 2$ are integers. Let $\cay(\ZZZ_n,S)$ be a connected circulant graph with order $n$ and degree $p^\ell - 1$ such that $S_0$ is pyramidal. Let $\iota(S_0): H_0 < H_1 < \cdots < H_{2t-1} \leq H_{2t} = \ZZZ_n$ be the longest admissible subgroup series of $\ZZZ_n$ associated with $S_0$ (see \eqref{eq:iotaS}) and $\mathbf{h}(S_0) = (h_0,h_1,\ldots, h_{t-1}, h_t)$ and $\mathbf{k}(S_0) = (k_0,k_1,\ldots, k_{t-1}, k_t)$ the corresponding sequences (see \eqref{eq:hk}). Let
\[
J_i:=\{jp^{h_{t-i-1}}k_{t-i}: 0 \le j \le (k_{t-i-1}/k_{t-i})-1\},\;\, 1\leq i\leq t-1.
\]
Set
$$
D_0 := \{jp^{h_{t-1}}:0\leq j\leq k_{t-1}-1\}
$$
and
\begin{equation}
\label{eq:D}
D_{2i-1}:=\{x + \tau_i(x): x \in D_{2i-2}\},\, D_{2i}:=D_{2i-1}+J_i,\;\, 1\leq i\leq t-1,
\end{equation}
where $\tau_i$ is a map from $D_{2i-2}$ to $p^{h_{t-i}}k_{t-i}\ZZZ_{p^{h_{t-i-1}-h_{t-i}}}$ satisfying $\tau_i(0)=0$. Note that for $1\leq i\leq t-1$, $|D_{2i-1}|=|D_{2i-2}|=k_{t-i}$ and both $D_{2i-1}$ and $D_{2i}$ depend on $D_{2i-2}$ and the choice of $\tau_i$. In addition, if $S_0$ is periodic with subgroup of periods $\langle p^{h_0}k_0\rangle$, then set $D_{2t-1}=\{x + \tau(x): x \in D_{2t-2} \}$, where $\tau$ is a map from $D_{2t-2}$ to $p^{h_0} k_0\ZZZ_{p^{\ell-h_0}}$ satisfying $\tau(0)=0$.

Note that both $D_{2t-2}$ and $D_{2t-1}$ contain $0$. Note also that $D_{2t-2}$ depends on $(\tau_{1}, \ldots, \tau_{t-1})$ and $D_{2t-1}$ depends on $(\tau_{1}, \ldots, \tau_{t-1}, \tau)$.
\qed
\end{construction}

\begin{lemma}
\label{co:co3.2}
Let $n=p^\ell m$, where $p$ is a prime and $\ell\geq1, m\geq 2$ are integers. Let $\Gamma=\cay(\ZZZ_n, S)$ be a connected circulant graph with order $n$ and degree $p^\ell-1$ such that $S_0$ is pyramidal. Let $2t$ be the length of the longest admissible subgroup series of $\ZZZ_n$ associated with $S_0$. The the following hold:
\begin{itemize}
\item[(a)]
if $S_0$ is aperiodic, then $D_{2t-2}$ is a perfect code in $\Gamma$; these perfect codes $D_{2t-2}$ are different for different choices of $(\tau_1, \ldots,\tau_{t-1})$;
\item[(b)]
if $S_0$ is periodic, then $D_{2t-1}$ is a perfect code in $\Gamma$; these perfect codes $D_{2t-1}$ are different for different choices of $(\tau_1, \ldots,\tau_{t-1}, \tau)$.
\end{itemize}
\end{lemma}

\begin{proof}
We use the notation in Construction \ref{const:DD} and set $\Gamma_i = \cay(\ZZZ_n/H_i, (S_0/H_i)\setminus \{H_i\})$, where $H_i$ is obtained from Construction \ref{const:1}, for $0 \le i \le 2t$.

We prove the first statements in (a) and (b) by induction on $t$. Consider the base case $t=1$ first. If $S_0$ is aperiodic, then $H_0=\{0\}$ and by (T3), $\Gamma_1$ is a complete graph and $0+H_1$ is a perfect code in $\Gamma_1$. By (T1) and Lemma \ref{le:ledeng1424}, $H_1$ is the unique prefect code in $\Gamma$ containing $0$. Note that $H_1 = \{0, p^{h_0},\ldots,  (m-1)p^{h_0}\} = D_0$. Hence $D_0$ is the unique perfect code in $\Gamma$ containing $0$. Now assume $S_0$ is $H_0$-periodic. Then $H_0\neq\{0\}$ and $\Gamma_1$ is a circulant graph whose extended connection set $S_0/H_1$ is an aperiodic pyramidal set of $\ZZZ_n/H_1$. Applying what we just proved to $\Ga_1$, we obtain that $D_0$ is the unique perfect code in $\Gamma_1$ containing $0$. Moreover, $D_0=\{0, p^{h_0},\ldots, (m-1)p^{h_0}\}$. Since $\Gamma\cong \Gamma_0\otimes \cay(\ZZZ_{p^{\ell-h_0}}, \ZZZ_{p^{\ell-h_0}}\setminus \{0\})$, by the definition of the wreath product we see that $D_1=\{x + \tau_1(x): x \in D_0\}$ is a perfect code in $\Gamma$ containing $0$, where $\tau_1$ is a map from $D_0$ to $p^{h_0}k_0\ZZZ_{p^{\ell-h_0}}$ satisfying $\tau_1(0)=0$. Since $D_0$ is the unique perfect code in $\Gamma_0$ containing $0$, by the definition of the wreath product, all perfect codes in $\Gamma$ containing $0$ can be obtained by different choices of $\tau_1$. So each perfect code in $\Gamma$ containing $0$ can be constructed from $D_1$. This establishes the result in the base case $t = 1$.

Suppose that the result is true when $t < l$ for some $l \ge 2$. Consider the case when $t=l$. Assume $S_0$ is aperiodic first. Then $\Gamma_2$ is a circulant graph whose extended connection set $T_0 := S_0/H_2$ is an aperiodic pyramidal set of $\ZZZ_n/H_2$. It can be verified that the longest admissible subgroup series associated with $T_0$ is $\iota(T_0): H_2/H_2 < H_3/H_2 < \cdots < H_{2l}/H_2 = \ZZZ_{n}/H_2$ and the corresponding sequences are $\mathbf{h}(T_0) = (a_0,a_1, \ldots,a_{l-1})$ and $\mathbf{k}(T_0) = (b_0,b_1,\ldots,b_{l-1})$, where $a_i = h_{i+1}, b_i = k_{i+1}$ for $0 \le i \le l-1$. By our hypothesis, $D_{2l-4}$ is a perfect code in $\Gamma_2$ containing $0$. Since $\Gamma_1\cong \Gamma_2\otimes \cay(\ZZZ_{p^{h_0-h_1}}, \ZZZ_{p^{h_0-h_1}}\setminus \{0\})$, $D_{2l-3}=\{x + \tau_{l-1}(x): x \in D_{2l-4}\}$ (see \eqref{eq:D}) is a perfect code in $\Gamma_1$ containing $0$, where $\tau_{l-1}$ is a map from $D_{2l-4}$ to $p^{h_1}k_1\ZZZ_{p^{h_0-h_1}}$ satisfying $\tau_{l-1}(0)=0$. By (T1) and Lemma \ref{le:ledeng1424}, $D_{2l-3}$ gives rise to a unique perfect code in $\Gamma$ containing $0$, namely $D_{2l-2}=D_{2l-3}+J_{l-1}=D_{2l-3}+\{jp^{h_0}k_1:0\leq j\leq (k_0/k_1)-1\}$.

Now assume $S_0$ is periodic. Then $\Gamma_0$ is a connected circulant graph whose extended connection set $S_0/H_0$ is an aperiodic pyramidal set of $\ZZZ_n/H_0$. By our hypothesis, $D_{2l-2}$ is a perfect code in $\Gamma_0$. Since $\Gamma\cong \Gamma_0\otimes \cay(\ZZZ_{p^{\ell-h_0}}, \ZZZ_{p^{\ell-h_0}}\setminus \{0\})$, by the definition of the wreath product, $D_{2l-1}=\{x + \tau(x): x \in D_{2l-2}\}$ is a perfect code in $\Gamma$ containing $0$, where $\tau$ is a map from $D_{2l-2}$ to $p^{h_0}k_0\ZZZ_{p^{\ell-h_0}}$. By mathematical induction, we have completed the proof of the first statements in (a) and (b).

Next, we prove the second statement in (a). Assume $S_0$ is aperiodic. Let $D_{2t-2}$ and $D'_{2t-2}$ be perfect codes in $\Gamma$ obtained from $(\tau_1, \ldots, \tau_{t-1})$ and $(\tau'_1, \ldots, \tau'_{t-1})$, respectively. It suffices to prove that $D_{2t-2}=D'_{2t-2}$ implies $(\tau_1, \ldots, \tau_{t-1}) = (\tau'_1, \ldots, \tau'_{t-1})$. Note that, for $0<2i-1\leq 2t-1$ and $x\in D_{2i-1}$ we have $0\leq x\leq p^{h_{t-i-1}}k_{t-i}-p^{h_{t-1}}$, and for $0\leq 2i<2t-1$ and $x\in D_{2i}$ we have $0\leq x\leq p^{h_{t-i-1}}k_{t-i-1}-p^{h_{t-1}}$. Note also that $D_{2t-2}=D_{2t-3}+J_{t-1}=\{x+jp^{h_0}k_1:x\in D_{2t-3}, 0\le j\le (k_{0}/k_{1})-1\}$ and $D'_{2t-2}=D'_{2t-3}+J_{t-1}=\{x+jp^{h_0}k_1:x\in D'_{2t-3}, 0\le j\le (k_{0}/k_{1})-1\}$.
Suppose that $D_{2t-2}=D'_{2t-2}$. Then for each $x+j_1p^{h_0}k_1\in D_{2t-2}$ there exists $y+j_2p^{h_0}k_1\in D'_{2t-2}$ such that $x+j_1p^{h_0}k_1=y+j_2p^{h_0}k_1$. Since $0\leq x, y\leq p^{h_0}k_1-p^{h_{t-1}}$, we have $-p^{h_0}k_1 < x-y < p^{h_0}k_1$. Thus $x+j_1p^{h_0}k_1=y+j_2p^{h_0}k_1$ (that is, $x-y=(j_2-j_1)p^{h_0}k_1$) if and only if $x=y$ and $j_1=j_2$. Meanwhile, we also obtain $D_{2t-3}=D'_{2t-3}$. Since $D_{2t-3}=\{x+\tau_{t-1}(x):x\in D_{2t-4}\}$ and $D'_{2t-3}=\{y+\tau'_{t-1}(y):y \in D'_{2t-4}\}$, this means that for each $x+\tau_{t-1}(x)\in D_{2t-3}$ there exists $y+\tau'_{t-1}(y)\in D'_{2t-3}$ such that $x+\tau_{t-1}(x)=y+\tau'_{t-1}(y)$. Since $0\leq x, y\leq p^{h_1}k_1-p^{h_{t-1}}$, we have $-p^{h_1}k_1< x-y < p^{h_1}k_1$. On the other hand, $\tau'_{t-1}(y)-\tau_{t-1}(x)$ is a multiple of $p^{h_1}k_1$ as both $\tau_{t-1}$ and $\tau'_{t-1}$ are maps from $D_{2t-4}$ to $p^{h_{1}}k_1\ZZZ_{p^{h_2}-h_1}$. Hence $x+\tau_{t-1}(x)=y+\tau'_{t-1}(y)$ (that is, $x-y=\tau'_{t-1}(y)-\tau_{t-1}(x)$) if and only if $x=y$ and $\tau_{t-1}(x)=\tau'_{t-1}(y)$. By the arbitrariness of $x\in D_{2t-4}$, it follows that $D_{2t-4}=D'_{2t-4}$ and $\tau_{t-1}=\tau'_{t-1}$. In the same vein as the proof that $D_{2t-2}=D'_{2t-2}$ implies $D_{2t-3}=D'_{2t-3}$, we can prove that $D_{2t-4}=D'_{2t-4}$ implies $D_{2t-5}=D'_{2t-5}$. Similarly to the proof that $D_{2t-3}=D'_{2t-3}$ implies $D_{2t-4}=D'_{2t-4}$ and $\tau_{t-1}=\tau'_{t-1}$, we can prove that $D_{2t-5}=D'_{2t-5}$ implies $D_{2t-6}=D'_{2t-6}$ and $\tau_{t-2}=\tau'_{t-2}$. By induction, we can prove that $D_{i}=D_{i}$ for $1\leq i\leq 2t-2$ and $\tau_{i}=\tau'_i$ for $1\le i\le t-1$. This proves the second statement in (a).

Finally, we prove the second statement in (b). Assume $S_0$ is periodic. Let $D_{2t-1}$ and $D'_{2t-1}$ be  perfect codes in $\Gamma$ obtained from $(\tau_1, \ldots, \tau_{t-1},\tau)$ and $(\tau'_1, \ldots, \tau'_{t-1},\tau')$, respectively. It suffices to prove that $D_{2t-1}=D'_{2t-1}$ implies $(\tau_1, \ldots, \tau_{t-1},\tau) = (\tau'_1, \ldots, \tau'_{t-1},\tau')$. In fact, since $\{x+\tau(x):x\in D_{2t-2}\} = D_{2t-1} = D'_{2t-1}=\{y+\tau'(y):y \in D'_{2t-2}\}$, for each $x+\tau(x)\in D_{2t-1}$ there exists $y+\tau'(y)\in D'_{2t-1}$ such that $x+\tau(x)=y+\tau'(y)$. Since $0\leq x, y\leq p^{h_0}k_0-p^{h_{t-1}}$, we have $-p^{h_0}k_0 < x-y < p^{h_0}k_0$. On the other hand, $\tau'(y)-\tau(x)$ is a multiple of $p^{h_0}k_0$ as both $\tau$ and $\tau'$ are maps from $D_{2t-2}$ to $p^{h_0}k_0\ZZZ_{p^{\ell-h_0}}$. Hence $x+\tau(x)=y+\tau'(y)$ if and only if $x=y$ and $\tau(x)=\tau'(y)$. By the arbitrariness of $x$, it follows that $D_{2t-2}=D'_{2t-2}$ and $\tau=\tau'$. Based on this and similarly to the argument in the previous paragraph, we can prove that $\tau_i=\tau'_i$ for $1\leq i\leq t-1$ and $\tau=\tau'$, as desired to establish the second statement in (b).
\end{proof}

We are now ready to prove Theorem \ref{cor:cor2.2}. Denote by $c(\Ga)$ the number of distinct perfect codes containing $0$ in a circulant graph $\Ga$.

\smallskip

\begin{Proof}~\ref{cor:cor2.2}.
Let $n=p^\ell m$, where $p$ is a prime and $\ell\geq 1, m\geq 2$ are integers, and let $\Gamma=\cay(\ZZZ_n,S)$ be a connected circulant graph with order $n$ and degree $p^\ell-1$ such that its extended connection set $S_0$ is pyramidal. As shown in Lemma \ref{lem:longest}, the series $\iota(S_0): H_0 < H_1 < H_2 < \cdots < H_{2t-1} \leq H_{2t} = \ZZZ_n$ (see \eqref{eq:iotaS}) is the longest admissible subgroup series of $\ZZZ_n$ associated with $S_0$. Let $\mathbf{h}(S_0) = (h_0, h_1, \ldots, h_{t-1}, h_t)$ and $\mathbf{k}(S_0) = (k_0, k_1, \ldots, k_{t-1}, k_t)$ be the associated sequences (see \eqref{eq:hk}) so that $H_{2i}=\langle p^{h_i}k_i\rangle$ and $H_{2i-1}=\langle p^{h_{i-1}}k_{i}\rangle$ for $1 \le i \le t$. Set $\Gamma_i=\cay(\ZZZ_n/H_i, (S_0/H_i)\setminus \{H_i\})$ for $0\leq i\leq 2t$. We aim to prove $c(\Ga) \ge p^{\left(\sum_{i=0}^{t-1}(h_{i-1}-h_{i})k_{i}\right)+h_{t-1}-\ell}$, where $h_{-1} = \ell$.

We proceed by induction on $t$. Consider the base case $t=1$ first. If $H_0=\{0\}$ (that is, $h_0=\ell$), then $S_0$ is aperiodic, $\Gamma_1$ is a complete graph by (T3), $0+H$ is the only perfect code in $\Gamma_1$ containing $0$, and $H_1$ is the only perfect code in $\Gamma$ containing $0$. Hence $c(\Gamma) = 1 = p^{0} = p^{h_{-1}-h_0}$. If $H_0\neq\{0\}$ (that is, $h_0<\ell$), then applying what we just proved to $\Ga_0$ yields $c(\Gamma_{0})=1$. Note that $\Gamma=\Gamma_0\otimes\cay(H_0,H_0\setminus\{0\})$ and every element of $H_0$ forms a perfect code in the complete graph $\cay(H_0,H_0\setminus\{0\})$. So, by Lemma \ref{co:co3.2}, we can choose different elements of $H_0$ to obtain different perfect codes in $\Gamma$. Since we only count perfect codes in $\Gamma$ containing $0$, it follows that $c(\Gamma)=c(\Gamma_0)p^{(\ell-h_0)(k_0-1)}=p^{(\ell-h_0)(k_0-1)}$. So we have proved that in the case when $t=1$ there are at least $p^{(h_{-1} - h_0) k_0 + h_0-\ell}$ perfect codes in $\cay(\ZZZ_n, S)$ containing $0$ no matter whether $H_0=\{0\}$ or $H_0\neq\{0\}$, where $h_{-1}=\ell$.

Suppose that the result is true when $t < l$ for some $l \geq 2$. We are going to prove that it is true when $t=l$.

\smallskip
\textsf{Case 1.} $H_0=\{0\}$ (that is, $h_0=\ell$).
\smallskip

Note that the extended connection set $T_0 := S_0/H_2$ of $\Ga_2$ is a pyramidal set of $\ZZZ_n/H_2$ for which the longest admissible subgroup series is $L_0 < L_1 < \cdots < L_{2(l-1)} = \ZZZ_n/H_2$, where $L_{i}=H_{i+2}/H_2$ for $0\leq i\leq 2(l-1)$. The corresponding sequences are $\mathbf{h}(T_0) = (a_0,a_1,\ldots,a_{l-1})$ and $\mathbf{k}(T_0) = (b_0,b_1,\ldots,b_{l-1})$, where $a_{i}=h_{i+1}$ and $b_{i}=k_{i+1}$ for $0\leq i\leq l-1$. By the induction hypothesis, $c(\Gamma_2) \ge p^{\left(\sum_{i=0}^{l-2}(a_{i-1}-a_i)b_i\right)+a_{l-2}-a_0}$. Note that $a_{-1}=a_0$ and $h_{-1}=h_0=\ell$. Similarly to the case $t=1$, since $\Gamma_1=\Gamma_2\otimes \cay(\ZZZ_{p^{h_0-h_1}},\ZZZ_{p^{h_0-h_1}}\setminus\{0\})$, we have
\begin{align*}
c(\Gamma_1) & = c(\Gamma_2)p^{(h_0-h_{1})(k_1-1)}\\
& \ge p^{\left(\sum_{i=0}^{l-2}(a_{i-1}-a_i)b_i\right)+a_{l-2}-a_0}\cdot p^{(h_0-h_{1})(k_1-1)}\\
& = p^{(a_{-1}-a_{0})b_0+\left(\sum_{i=1}^{l-2}(a_{i-1}-a_i)b_i\right)+(h_0-h_1)k_1+a_{l-2}-a_0-h_0+h_1}\\
& = p^{\left(\sum_{i=1}^{l-2}(h_{i}-h_{i+1})k_{i+1}\right)+(h_0-h_1)k_1+h_{l-1}-h_0} \\
& = p^{(h_{-1}-h_0)k_0+\left(\sum_{i=1}^{l-2}(h_{i}-h_{i+1})k_{i+1}\right)+(h_0-h_1)k_1+h_{l-1}-\ell}\\
& = p^{\left(\sum_{i=0}^{l-1}(h_{i-1}-h_{i})k_{i}\right)+h_{l-1}-\ell}.
\end{align*}
Here we used $a_0=h_1$ and $a_{l-2}=h_{l-1}$ to obtain the fourth line. By Lemma \ref{co:co3.2}, we have $c(\Gamma)=c(\Gamma_0)=c(\Gamma_1)$ and the desired result follows.

\smallskip
\textsf{Case 2.} $H_0\neq\{0\}$ (that is, $h_0<\ell$).
\smallskip

The extended connection set $T_0 := S_0/H_0$ of $\Gamma_0$ is a pyramidal set of $\ZZZ_n/H_0$ for which the longest admissible subgroup series is $L_0<L_1<\cdots<L_{2l}=\ZZZ_n/H_0$, where $L_i=H_i/H_0$ for $1\le i\le 2l$. The corresponding sequences are $\mathbf{h}(T_0) = (a_0,a_1,\ldots,a_{l})$ and $\mathbf{k}(T_0) = (b_0,b_1,\ldots,b_{l})$, where $a_i=h_i$ and $b_i=k_i$ for $1\le i\le l$.
Applying what we proved in Case 1 to $\Ga_0$, we obtain  $c(\Gamma_0)\ge p^{\left(\sum_{i=0}^{l-1}(a_{i-1}-a_{i})b_{i}\right)+a_{l-1}-a_0}$, where $a_{-1}=h_0$. Similarly to the case $t=1$, we have
\begin{align*}
c(\Gamma) & = c(\Gamma_0)p^{(\ell-h_{0})(k_0-1)}\\
& \ge p^{\left(\sum_{i=0}^{l-1}(a_{i-1}-a_i)b_i\right)+a_{l-1}-a_0} \cdot p^{(\ell-h_{0})(k_0-1)}\\
& =
p^{\left((a_{-1}-a_0)b_0+\sum_{i=1}^{l-1}(a_{i-1}-a_i)b_i\right)+a_{l-1}-a_0+(\ell-h_{0})k_0-\ell+h_0}\\
& =
p^{\sum_{i=1}^{l-1}(a_{i-1}-a_i)b_i+(\ell-h_0)k_0+a_{l-1}-a_0-\ell+h_0}\quad(\text{as $a_{-1}=a_0=h_0$})\\
& =
p^{\sum_{i=1}^{l-1}(h_{i-1}-h_i)k_i+(\ell-h_0)k_0+h_{l-1}-\ell}\quad(\text{as $a_{i}=h_i$ and $b_i=k_i$ for $0\le i\le l$})\\
& =
p^{\sum_{i=1}^{l-1}(h_{i-1}-h_i)k_i+(h_{-1}-h_0)k_0+h_{l-1}-\ell}\quad(\text{taking $h_{-1}=\ell$})\\
& =
p^{\left(\sum_{i=0}^{l-1}(h_{i-1}-h_i)k_i\right)+h_{l-1}-\ell}.
\end{align*}
So the result holds for $\Ga$.

Combining the two cases above, we have proved that the result is true when $t=l$. This completes the proof by mathematical induction.
\qed
\end{Proof}

\section{Final remarks}\label{sec:final}

The main result of this paper gives a complete explicit description of the family of circulant graphs admitting a perfect code with extended connection set of cardinality a prime power (Theorem \ref{th:th2.2.1}) as well as a structural description of such circulant graphs (Theorem \ref{th:thplp}). The characterization also leads to a counting result (Theorem \ref{cor:cor2.2}) which shows that the universe of perfect codes in circulant graphs is relatively rich.

 The characterization can be extended by a similar approach to the case of perfect codes in the directed case, namely when the connection set $S$ of $\cay(\ZZZ_n, S)$ is not symmetric, although the corresponding results are more involved and we have not attempted to report them here. In another direction, a result analogous to Theorem \ref{th:thplp} can be obtained for the closely related family of Cayley sum graphs $\cay^{+}(\ZZZ_n, S)$ of cyclic groups, where two vertices $x,y \in \ZZZ_n$ are adjacent if and only if $x+y\in S$. 

The main limitation of our approach is the restriction to prime powers for the cardinality of the extended connection sets. This is due to the use of the key Lemma \ref{thm:primepower}, which provides the strategy to reach the pyramidal structure in  Theorem \ref{th:thplp} and then to the recursive description of  the family of circulant graphs admitting perfect codes by lifts of the complete graph given in Construction \ref{const:3}. It is still an open problem to determine the extent to which the statement of Lemma \ref{thm:primepower} can be generalized. In view of the results in the present paper it appears that a complete characterization of perfect codes, even in the case of circulant graphs, beyond the prime power case may become a project of high complexity. In any case, in absence of the result in  Lemma \ref{thm:primepower}, new ideas and techniques will be required.

\small {

}

\newpage

\appendix

\section{A formula}
\label{sec:form}

In this appendix we give a formula for the number of connected circulant graphs $\cay(\ZZZ_n, S)$ with order $n$ and degree $p^\ell-1$ such that $(S_0-S_0)\cap\langle p^\ell \rangle=\{0\}$. Note that, by Theorem \ref{th:thplp}, such a graph $\cay(\ZZZ_n, S)$ admits $\langle p^\ell \rangle$ as a perfect code if in addition $S_0$ is pyramidal and aperiodic. To obtain our formula we need the following two lemmas.

\begin{lemma}
\label{le:le2.2.1}
Let $n=p^\ell m$ be an integer, where $p$ is a prime and $\ell \geq1, m \geq 2$ are integers. Let $\Ga = \cay(\ZZZ_n,S)$ be a circulant graph with order $n$ and degree $p^{\ell}-1$ such that $(S_0-S_0)\cap\langle p^\ell\rangle=\{0\}$. If $\Ga \in \GG_{\mathbf{l}, \mathbf{m}}(n, p^\ell)$ for some sequences $\mathbf{l} = (\ell_1, \ell_2, \ldots, \ell_t), \mathbf{m} = (m_1, m_2, \ldots, m_t)$ of positive integers satisfying $\sum_{i=1}^t \ell_i = \ell$ and $\prod_{i=1}^{t}m_i = m$, then $p$ is not a divisor of $m_i$ for $2 \le i \le t$.
\end{lemma}

\begin{proof}
Let $r = \ell_t$. Suppose, for a contradiction, that $p$ is a divisor of $m_i$ for some $i$ between $2$ and $t$. Choose $k$ to be the largest subscript such that $p$ divides $m_k$. Then $2 \le k \le t$ and $p$ is not a divisor of $m_i$ for $k < i \leq t$. For $1\leq i\leq t$, let $\mathbf{l}_i = (\ell_i, \ldots, \ell_t)$, $\mathbf{m}_i = (m_i, \ldots, m_t)$, $\ell(i) = \ell_{i}+\cdots+\ell_{t}$, and $n(i) = p^{\ell_{i}+\cdots+\ell_{t}}m_{i}\cdots m_t = p^{\ell(i)}m_{i}\cdots m_t$. Let $\Gamma_i = \cay(\ZZZ_{n(i)}, S(i))$ be a circulant graph in $\GG_{\mathbf{l}_i, \mathbf{m}_i}(n(i), p^{\ell(i)})$. In particular, $\Ga_k \in \GG_{\mathbf{l}_k, \mathbf{m}_k}(n(k), p^{\ell(k)})$. By part (c) of Lemma \ref{pr:pr4.3}, we have
\[
\GG_{\mathbf{l}_k, \mathbf{m}_k}(n(k), p^{\ell(k)})
\subseteq (\baf_{m_{k}}\cdots \baf_{m_t})(\bag_{p^{\ell_{k}}}\cdots\bag_{p^{\ell_{t-1}}}(K_{p^{\ell_t}})).
\]
Moreover, by part (b) of Lemma \ref{pr:pr4.3}, we have $(\bag_{p^{\ell_{k}}}\cdots\bag_{p^{\ell_t-1}})(K_{p^{\ell_t}})=K_{p^{\ell(k)}}$. So, by part (a) of Lemma \ref{pr:pr4.3}, $\Gamma_k \in (\baf_{m_{k}}\cdots \baf_{m_t})(K_{p^{\ell(k)}}) = \baf_{m_{k} \cdots m_t }(K_{p^{\ell(k)}})$. Note that $K_{p^{\ell(k)}} = \cay(\ZZZ_{p^{\ell(k)}},\ZZZ_{p^{\ell(k)}}\setminus\{0\})$ and $x-y \not\equiv 0 \pmod{p^{\ell(k)}}$ for distinct $x, y \in \ZZZ_{p^{\ell(k)}}$. By part (a) of Lemma \ref{pr:pr4.1}, we have $x-y \not\equiv 0\pmod{p^{\ell(k)}}$ for distinct $x, y \in S(k)_0$. Let $T(\sigma)_0\in \bag_{p^{\ell_{k-1}}}(S(k)_0)$. Then $T(\sigma)_0=\{s+in(k): s\in S(k)_0, 0\leq i\leq p^{\ell_{k-1}}-1\}$, where $n(k)$ is as above. Since $p$ is a divisor of $m_{k}$ and $\ell_{k-1} + \ell(k)= \ell(k-1)$, $p^{\ell_{k-1}-1}n(k)$ is divisible by $p^{\ell(k-1)}$. Thus, by part (c) of Lemma \ref{pr:pr4.1}, there are distinct elements $x, y \in S_0$ such that $x - y\equiv 0\pmod{p^\ell}$, that is, $(S_0-S_0)\cap \langle p^\ell\rangle\neq \{0\}$, which contradicts our assumption.
\end{proof}

For positive integers $d, n$ and prime $p$, let $\nu(d)$ be the number of prime divisors of $d$ and $D_n^p$ the set of square-free divisors of $n$ greater than $1$ and coprime to $p$.

\begin{lemma}
\label{le:legcd}
Let $n=p^\ell m$, where $p$ is a prime and $\ell>1, m\ge 2$ are integers such that $\ell+p\ge 4$.
 If $D_n^p\neq\emptyset$, then the number of elements $T(\sigma) \in \baf_{m}(K_{p^{\ell}})$ satisfying $\gcd(T(\sigma))\neq 1$ are exactly $\sum\limits_{d\in D_n^p}(-1)^{\nu(d)-1}(m/ d)^{\frac{p^\ell-1}{2}}$.
\end{lemma}

\begin{proof}
For each integer $1\leq i\leq p^\ell-1$, set
\[P_i=\{x\in\ZZZ_{n}:x\equiv i\pmod{p^\ell}\}=\{i, i+p^\ell,\ldots,i+(m-1)p^\ell\}.\]
Note that for any feasible map $\sigma$ from $\ZZZ_{p^\ell}\setminus\{0\}$ to $\{0,1,\ldots,(n/p^\ell)-1\}$, $T(\sigma)=\{x+\sigma(x)p^\ell:x\in\ZZZ_{p^\ell}\setminus\{0\}\}$. Thus, for each integer $1\leq i\leq p^\ell-1$, $|T(\sigma)\cap P_i|=1$.
View $\ZZZ_{p^\ell}$ as a set.
Since $p$ is not a divisor of $\gcd(\ZZZ_{p^{\ell}}\setminus\{0\})$, $p$ is not a divisor of $\gcd(T(\sigma))$.

For any integer $d\in D_n^p$, if $\gcd(p,d)=1$, then $\gcd(p^\ell,d)=1$. There are two integers $a, b$ such that $a d+b p^\ell=1$, and so $ad=-bp^\ell+1$. Then for any integer $1\leq i\leq p^\ell-1$, there is an integer $\alpha$ between $0$ and $d-1$ such that $d$ is a divisor of $i+\alpha p^{\ell}$ and for any integer $0\leq \beta\leq \alpha-1$, $d$ is not a divisor of $i+\beta p^{\ell}$.
Clearly, for each integer $j\in \{0,1,\ldots,(m/d)-1\}$, $d$ is also a divisor of $i+(jd+\alpha)p^\ell$.
Thus for each integer $1\le i\le p^\ell-1$ there are exactly $m/d$ elements in $P_i$ such that $d$ is a divisor of these $m$ elements. Moreover, since $T(\sigma)$ is inverse-closed in $\ZZZ_n$, if $s$ is fixed, then $-s$ is also fixed. Hence for each divisor $d\in D_n^p$ there are exactly $(m/d)^{\frac{p^\ell-1}{2}}$ connection sets $T(\sigma)$ in $\baf_{m}(K_{p^\ell})$ such that $d$ is a divisor of $\gcd(T(\sigma))$.

Suppose $m=p^{a(0)}p^{a(1)}_1p^{a(2)}_2\cdots p^{a(b)}_b$, where $a(0)\geq 0$, $p, p_1, p_2, \ldots, p_b$ are different primes and for any integer $1\le i\le b$, $a(i)\ge 1$. For any integer $1\le i\le b$, define $W(i)=\{p_{1(1)}p_{1(2)}\cdots p_{1(i)}: p_{1(1)}, p_{1(2)},\ldots, p_{1(i)} \text{ are different elements of } \{p_1, p_2, \ldots, p_b\}\}$.
For each integer $1\le i\le p^\ell-1$, define
\[
U(i)=\{T(\sigma)_0\in \baf_m(K_{p^\ell}): \text{there exists $d\in W(i)$ such that $d$ is a divisor of $\gcd(T(\sigma))$} \}.
\]
Let $F(i)=\sum_{d\in W(i)}(m/d)^{\frac{p^\ell-1}{2}}$. Note that $U(b)\subseteq U(b-1)\subseteq\cdots\subseteq U(1)$. Thus each extended connection set $T(\sigma)_0$ in $\baf_m(K_{p^\ell})$  satisfying $\gcd(T(\sigma))\neq 1$  is counted by $F(1)$. Moreover, it is known that for any integer $1\le i<j\le b$, $F(i)$ counts extended connection sets $T(\sigma)_0$ in $\baf_m(K_{p^\ell})$ satisfying that there is an integer $d\in W(j)$ such that $d$ is a divisor of $\gcd(T(\sigma))$ exactly $\binom{j}{i}$ times. Thus, in order to count the number of elements $T(\sigma)_0\in \baf_{m}(K_{p^{\ell}})$ satisfying $\gcd(T(\sigma))\neq1$, we only need to consider the number of extended connection sets $T(\sigma)_0\in \baf_{m}(K_{p^\ell})$ satisfying $\gcd(T(\sigma))\in D_n^p$. Since $\sum_{i=1}^{a}(-1)^{i-1}\binom{a}{i}=1$ for any integer $a \ge 1$, for any $1\le j\le b$, $\sum_{i=1}^j(-1)^{j-1}F(j)$ counts each extended connection set $T(\sigma)_0$ in $\baf_{m}(K_{p^\ell})$ satisfying $\gcd(T(\sigma))\in \cup_{i=1}^jW(j)$ exactly once. In particular, taking $j=b$, we obtain that the number of elements $T(\sigma) \in \baf_{m}(K_{p^{\ell}})$ satisfying $\gcd(T(\sigma))\neq 1$ is exactly $\sum\limits_{d\in D_n^p}(-1)^{\nu(d)-1}(m/ d)^{\frac{p^\ell-1}{2}}$.
\end{proof}

\begin{theorem}
\label{le:le5.2.2.2}
Let $n=p^\ell m$, where $p$ is a prime and $\ell\geq 1, m\geq 2$ are integers such that $\ell + p \ge 4$. Then the number of connected circulant graphs $\cay(\ZZZ_n, S)$ with order $n$ and degree $p^\ell-1$ such that $(S_0-S_0)\cap\langle p^\ell \rangle=\{0\}$ is equal to
\begin{equation}
\label{eq:ga1p}
\left(\frac{n}{p^{\ell}}\right)^{r_{\ell}}+\sum_{d\in D_{n}^p}(-1)^{\nu(d)}\left(\frac{n}{p^\ell d}\right)^{r_{\ell}},
\end{equation}
where $r_{\ell} = (p^\ell-1)/2$ if $p$ is odd and $r_{\ell} = 2^{\ell-2}$ if $p = 2$, with the understanding that the second term vanishes when $D_{n}^p=\emptyset$.
\end{theorem}

Denote by $N(n, p^\ell)$ the said number in the theorem. The condition $\ell + p \ge 4$ ensures $\ell \ge 2$ when $p = 2$. (If $\ell = 1$ and $p = 2$, then $N(n, p^\ell) = 0$ as there is no connected circulant graph with order $n \ge 4$ and degree $1$.) Note that when $D_{n}^2=\emptyset$ we have $n = 2^{\ell}$ and \eqref{eq:ga1p} yields $N(n, 2^\ell) = 1$.

\begin{proof}
By Theorem \ref{th:th2.2.1}, a connected circulant graph $\Gamma=\cay(\ZZZ_n,S)$ with order $n$ and degree $p^\ell-1$ admits a perfect code if and only if $\Ga \in \GG_{\mathbf{l}, \mathbf{m}}(n, p^\ell)$ for some $\mathbf{l} = (\ell_1, \ell_2, \ldots, \ell_t)$ and $\mathbf{m} = (m_1, m_2, \ldots, m_t)$ satisfying $\sum_{i=1}^t \ell_i = \ell$ and $\prod_{i=1}^{t}m_i = m$. If in addition $\Gamma$ satisfies  $(S_0-S_0)\cap\langle|S_0|\rangle=\{0\}$, then by Lemma \ref{le:le2.2.1} we have $\gcd(m_i, p)=1$ for $2\le i\le t$. By parts (a)--(c) of Lemma \ref{pr:pr4.3}, we have $\Gamma\in \baf_{n/p^{\ell}}(K_{p^{\ell}})$. Note that any element of $\baf_{n/p^{\ell}}(K_{p^\ell})$ is of the form $T(\sigma) = \{x + \sigma(x)p^{\ell}: x \in \ZZZ_{p^\ell} \setminus \{0\}\}$, where $\s$ is a feasible map from $\ZZZ_{p^\ell}$ to $\{0, 1, \ldots, (n/p^{\ell})-1\}$ (so that $\s(0) = 0$ and $-T(\sigma) = T(\sigma)$ in $\ZZZ_n$).

Assume $p$ is odd. Since for each $x\in \ZZZ_{p^\ell}\setminus\{0\}$, $\sigma(-x)$ is determined by $\sigma(x)$, each $\sigma \in \baf_{n/p^{\ell}}(K_{p^{\ell}})$ is determined by $\sigma(x)$ for $x \in \{1, 2, \ldots, (p^\ell-1)/2\} \subset \ZZZ_{p^{\ell}}\setminus\{0\}$. Since each $\sigma(x)$ can independently take any value from $\{0, 1, \ldots, (n/p^{\ell})-1\}$, it follows that the cardinality of $\baf_{n/p^{\ell}}(K_{p^{\ell}})$ is exactly $(n/p^{\ell})^{\frac{p^{\ell}-1}{2}}$. Since $\Gamma$ satisfies $(S_0-S_0)\cap\langle|S_0|\rangle=\{0\}$ and  $|S_0|=p^\ell$, for any $T(\sigma)\in \baf_{n/p^{\ell}}(K_{p^{\ell}})$ and $y \in T(\sigma)$, we have $y \not\equiv 0\pmod{p^\ell}$. By the definition of $T(\sigma)$, $\gcd(T(\sigma))$ is not divisible by $p$. If $D_n^p=\emptyset$, then $n/p^\ell=p^r$ for some integer $r\geq 0$. Since $\gcd(T(\sigma))$ is a divisor of $n=p^{\ell+r}$, by the definition of $T(\sigma)$ again we see that $\gcd(T(\sigma))=1$ and hence all circulant graphs in $\baf_{n/p^\ell}(K_{p^\ell})$ are connected. Hence, if $D_n^p=\emptyset$, then $N(n, p^\ell) = (n/p^{\ell})^{\frac{p^\ell-1}{2}}$. If $D_n^p\neq\emptyset$, then by Lemma \ref{le:legcd} the number of elements $T(\sigma) \in \baf_{n/p^{\ell}}(K_{p^{\ell}})$ satisfying $\gcd(T(\sigma))\neq 1$ is exactly $\sum\limits_{d\in D_n^p}(-1)^{\nu(d)-1}(n/(p^\ell d))^{\frac{p^\ell-1}{2}}$, and therefore $N(n, p^\ell)$ is given by \eqref{eq:ga1p}.

Now assume $p=2$. Then $\ell \ge 2$ by our assumption, and $n/2$ belongs to every $T(\sigma) \in \baf_{n/2^{\ell}}(K_{2^\ell})$. So there are exactly $(n/2^{\ell})^{2^{\ell-2}}$ different elements $T(\sigma) \in \baf_{n/2^{\ell}}(K_{2^\ell})$, all of which are such that $\gcd(T(\sigma))$ is odd. By Remark \ref{rem:fgmbar}, $\gcd(n/2^\ell,2)=1$. Thus, if $D_{n}^2=\emptyset$, then $n=2^\ell$ and hence $\baf_{n/2^{\ell}}(K_{2^{\ell}})$ contains only one graph, yielding $N(n, 2^\ell) = 1$. If $D_n^2\neq\emptyset$, then there are exactly $\sum\limits_{d\in D_n^2}(-1)^{\nu(d)-1}(n/(2^\ell d))^{2^{\ell-2}}$ different elements $T(\sigma) \in \baf_{n/2^\ell}(K_{2^\ell})$ such that $\gcd(T(\sigma))\neq 1$, and hence $N(n, 2^\ell)$ is given by \eqref{eq:ga1p}.
\end{proof}

\end{document}